\documentclass[12pt]{amsart}

\usepackage[colorlinks=true,urlcolor=blue,
bookmarks=true,bookmarksopen=true,citecolor=blue
]{hyperref}

\usepackage[all, cmtip]{xy}

\usepackage{graphicx, subfig}
\usepackage{epsfig}
\usepackage{amscd}
\usepackage[mathscr]{eucal}
\usepackage{amssymb}
\usepackage{amsxtra}
\usepackage{amsmath}
\usepackage{latexsym}
\usepackage[all]{xy}
\usepackage{enumerate}
\usepackage{mathrsfs}

\usepackage{epstopdf}

\usepackage{color}
\theoremstyle{plain}

\newtheorem{mainthm}{Theorem}
\newtheorem{maincor}[mainthm]{Corollary}
\newtheorem{thm}{Theorem}[subsection]

\newtheorem{lem}[thm]{Lemma}
\newtheorem{prop}[thm]{Proposition}

\theoremstyle{definition}
\newtheorem{dfn}[thm]{Definition}

\theoremstyle{remark}
\newtheorem{rem}[thm]{Remark}

\theoremstyle{plain}

\newcommand{\Id}{{{\mathchoice {\rm 1\mskip-4mu l} {\rm 1\mskip-4mu l}
      {\rm 1\mskip-4.5mu l} {\rm 1\mskip-5mu l}}}}
\newcommand{\cobto}{\leadsto}
\newcommand{\id}{\textnormal{id}}

\newcommand{\R}{\mathbb{R}}
\newcommand{\Z}{\mathbb{Z}}

\newcommand{\N}{\mathbb{N}}
\newcommand{\C}{\mathbb{C}}

\newcommand{\La}{\Lambda}

\newcommand{\cob}{\mathcal{C}ob}
\newcommand{\fuk}{\mathcal{F}uk}

\newcommand{\mor}{{\textnormal{Mor\/}}}

\newcommand{\en}{\mathcal{E}nd}

\newcommand{\ocaddress}{cornea@dms.umontreal.ca}

\begin{document}

\title[Categorification of Seidel's representation.]{Categorification of Seidel's representation.}
\date{\today}

\thanks{The first author was supported by a University of Montreal End of doctoral studies grant, and TIDY and Raymond and Beverly Sackler fellowships from Tel Aviv University.\\
The second author was supported by an NSERC Discovery grant
  and a FQRNT Group Research grant}

\author{%Paul Biran,
 Fran\c{c}ois Charette and Octav Cornea}

%\address{Paul Biran, Department of Mathematics, ETH-Z\"{u}rich,
 % R\"{a}mistrasse 101, 8092 Z\"{u}rich, Switzerland}
%\email{\pbaddress}
\address{Fran\c{c}ois Charette, Max Planck Institute for Mathematics, Vivatsgasse 7, 53111 Bonn, Germany} \email{charettf@mpim-bonn.mpg.de}
 \address{Octav Cornea, Department of Mathematics
  and Statistics University of Montreal C.P. 6128 Succ.  Centre-Ville
  Montreal, QC H3C 3J7, Canada} \email{\ocaddress}

%\bibliographystyle{plain}

% ----------------------------------------------------------------------
%
 %Abstract

 \begin{abstract}
Two natural symplectic constructions, the Lagrangian suspension and Seidel's quantum
representation of the fundamental group of the group of Hamiltonian
diffeomorphisms, $Ham(M)$, with $(M,\omega)$ a monotone symplectic
manifold, admit categorifications as actions of the fundamental
groupoid $\Pi(Ham(M))$ on a cobordism category recently introduced
in \cite{Bi-Co:cob2} and, respectively, on a monotone variant of the
derived Fukaya category. We show that the functor constructed in
\cite{Bi-Co:cob2} that maps the cobordism category to the derived
Fukaya category is equivariant with respect to these actions.
\end{abstract}

\maketitle

% ----------------------------------------------------------------------
%
% Beginning of text
%

% !TEX root = Serep.tex

\section{Introduction}

Let $(M^{2n},\omega)$ be a symplectic manifold which is closed or tame at infinity.

\

The main purpose of this paper is to  relate two important, basic
constructions in symplectic topology:  Lagrangian suspension \cite{Ar:cob-1,Ar:cob-2},\cite{Aud:calc-cob},\cite{Po:hambook} and
 Seidel's representation $S: \pi_{1}(Ham(M))\to QH(M)$ in its Lagrangian version
\cite{Se:pi1},\cite{Hu-La:Seidel-morph}, \cite{Hu-La-Le:monodromy}.

\

Our first remark is that, after convenient ``categorification'',  these two constructions
can be naturally viewed as  two actions of the fundamental groupoid $\Pi (Ham(M))$
on a cobordism category $\cob^{d}_{0}(M)$ (introduced  in
\cite{Bi-Co:cob1}) and, respectively, on the derived 
Fukaya category $D\fuk^{d}(M)$ of $M$.  We then show that these two actions 
are interchanged
by the functor:
$$\mathcal{F}:\cob^{d}_{0}(M)\to D\fuk^{d}(M)$$
introduced in \cite{Bi-Co:cob1} and \cite{Bi-Co:cob2}. 
In fact, we prove a slightly more refined statement:
 \begin{mainthm} \label{thm:main2}
The following diagram of action bifunctors commutes:
\begin{eqnarray}\label{eq:main-diag2}
   \begin{aligned}
\xymatrix{
    \Pi(Ham(M)) \times \cob^d_0(M) \ar[rrr]^{\Sigma} \ar[d]_{\id \times
\widetilde{\mathcal{F}}} & & & \cob^d_0(M) \ar[d]^{\widetilde{\mathcal{F}}}\\
    \Pi(Ham(M)) \times T^{S}D\fuk^{d}(M) \ar[rrr]^{\widetilde{S}}& & &
T^{S}D\fuk^{d}(M)}
\end{aligned}
\end{eqnarray}
\end{mainthm}

The category $T^{S}D\fuk^{d}(M)$ is a category associated to 
$D\fuk^{d}(M)$ by a purely algebraic process \cite{Bi-Co:cob2}. 
The morphisms in this category reflect the ways in
which the objects in $D\fuk^{d}(M)$ can be decomposed by iterated
exact triangles in $D\fuk^{d}(M)$ (which, we recall, is a
triangulated category). 
The functor $\widetilde{\mathcal{F}}$ - constructed in \cite{Bi-Co:cob2} -
is a refinement of $\mathcal{F}$ that ``sees'' the triangulated structure of $D\fuk^{d}(M)$.
The decorations ${\ }^{d}$ and ${\ }_{0}$ come from certain constraints that need 
to be imposed on the class of Lagrangians in use so that all these notions are well defined.

In the diagram, $\Sigma$ represents the action of $\Pi(Ham(M))$ on
$\cob^{d}_{0}(M)$ and is an extension of Lagrangian suspension.  
The action of $\Pi(Ham(M))$ on $T^{S}D\fuk^{d}(M)$ is denoted by
$\widetilde{S}$ and is a refinement, first to the derived Fukaya
category  $D\fuk^{d}(M)$ and then to $T^{s}D\fuk^{d}(M)$, of the
Lagrangian Seidel representation.

\

There is yet another perspective on the  commutativity in (\ref{eq:main-diag2})
 that is based on the equivalent definition of the action of
a (strict) monoidal category $\mathcal{M}$ on a category $\mathcal{C}$ as a strict
 monoidal functor $\mathcal{M}\to \en (\mathcal{C},\mathcal{C})$.
The commutativity of diagram (\ref{eq:main-diag2}) is then
equivalent to the commutativity of the bottom square of diagram
(\ref{eq:main-diag}) below:

\begin{maincor} \label{thm:main}
The following diagram of categories and functors commutes:
\begin{eqnarray}\label{eq:main-diag}
   \begin{aligned}
   \xymatrix{
     \pi_{1}(Ham(M))\ar[rrr]^{S} \ar[d]_{i}
     & &  &QH(M)^{\ast}\ar[d]^{\ast}\\
   \Pi (Ham(M))\ar[d]_{\Sigma}\ar[rrr]^{\widetilde{S}} & & &\en (T^{S}D\fuk^{d}(M))\ar[d]^{\widetilde{\mathcal{F}}^{\ast}}\\
    \en(\cob^{d}_{0}(M))\ar[rrr]_{\widetilde{\mathcal{F}}_{\ast}}& & & fun (\cob^{d}_{0}(M), T^{S}D\fuk^{d}(M))}
\end{aligned}
\end{eqnarray}
The categories and functors in the top square are strict monoidal as is the functor $\Sigma$.
\end{maincor}

Here $QH(M)^{\ast}$ are the invertible elements in the quantum
homology of $M$, the functor $S$ is Seidel's representation \cite{Se:pi1} viewed
as a monoidal functor and the action $\ast$ is a refinement of the
module action of quantum homology on Lagrangian Floer homology
\cite{Bi-Co:rigidity}. The functors $\widetilde{\mathcal{F}}_{\ast}$
and $\widetilde{\mathcal{F}}^{\ast}$ are  induced respectively by
composition and pre-composition with $\widetilde{\mathcal{F}}$.

\begin{rem}\label{rem:open-closed} A simpler form of diagram (\ref{eq:main-diag}) is yet another commutative diagram of
monoidal categories and functors, of the same shape,  but with the
more familiar $D\fuk^{d}(M)$ taking the place of
$T^{S}D\fuk^{d}(M)$. Further, $QH(M)$ replaces $QH(M)^{\ast}$ in the
upper right corner and $\mathcal{F}$, a linearization of the functor
$\widetilde{\mathcal{F}}$, takes the place of
$\widetilde{\mathcal{F}}$ in this simplified diagram.  In this form,
the action $\ast$ is closely related to a map sometimes called the
{\em closed-open map},  $\mathcal{CO}$. More precisely, $\ast$ is
the composition
\begin{eqnarray}\label{eq:HH}QH(M)  \stackrel{\mathcal{CO}}{\longrightarrow} HH(\fuk^{d}(M),\fuk^{d}(M))=H(\mor_{\en(\fuk^{d}(M))}(\id,\id)) \to \\ \nonumber \to \en (D\fuk^{d}(M)~.~
\end{eqnarray}
Here $HH(-,-)$ stands for Hochschild co-homology
\cite{Se:Homological}. A definition of this co-homology is provided
by the equality in (\ref{eq:HH}) above.  $\en(\fuk^{d}(M))$ stands
for the $A_{\infty}$-category of endofunctors of $\fuk^{d}(M)$ and
the last arrow in the composition is the restriction of the
canonical functor $H(\en(\fuk^{d}(M))\to \en(D\fuk^{d}(M))$. This
simplified diagram is of interest in itself even if it does not take
into account explicitly the triangulation of $D\fuk^{d}(M)$.
\end{rem}

The plan of the paper is as follows. In \S\ref{sec:ingred} we
describe the ingredients in diagram (\ref{eq:main-diag}).  This part
consists mostly of recalls  but there are also a few constructions
that seem not to be present explicitly in the literature and thus we
spend more time on these points. In particular, we show that the
usual notion of Lagrangian suspension extends to the action $\Sigma$
on the cobordism category. We also give a few details to justify
why the closed-open map gives rise to an action.  

The main part of
the paper is in \S\ref{sec:commut}. In \S\ref{subsec:bottom} we
verify the commutativity in diagram (\ref{eq:main-diag2}) and thus
we prove Theorem \ref{thm:main2}. There are two main ingredients in this  proof. The first consists in
some algebraic manipulations based on the properties of the functor $\widetilde{\mathcal{F}}$
from \cite{Bi-Co:cob2}. The second ingredient -  that first appeared in the first author's
thesis \cite{Cha:thesis} -  is the main geometric novelty brought  by the paper. 
Its content is as follows.  Given two Lagrangians $L$, $L'$ (under the usual technical
constraints) and  a path ${\bf g}_{t}$  of Hamiltonian diffeomorphisms there is
a natural ``moving boundary'' morphism
 $$CF(g_{0}(L), L')\to CF(g_{1}(L), L')$$ that induces in homology 
 Seidel's  Lagrangian morphism.  On the other hand,  
 the Lagrangian suspension of $L$ with respect to ${\mathbf{g}}$ is a cobordism from $g_{0}(L)$
 to $g_{1}(L)$. The functor $\widetilde{\mathcal{F}}$ associates to such a cobordism
 another morphism $CF(g_{0}(L),L')\to CF(g_{1}(L),L')$. The key geometric point is that these two
 morphisms are chain-homotopic.  
 
 In \S\ref{subsec:top} we sketch
the proof of the commutativity of the top square in diagram
(\ref{eq:main-diag}). This commutativity is known in various
settings by experts but  we include the  argument here so as to
complete the proof of Corollary \ref{thm:main}.

 The last section,
\S\ref{sec:examples}, mainly contains a few examples. We show how
to use the commutativity of (\ref{eq:main-diag}) to produce examples of cobordisms
with identical ends but that are not horizontally Hamiltonian isotopic.  In particular, we produce
Lagrangians $L\subset M$ and cobordisms $V\subset \C\times M$ diffeomorphic to $\R\times L$,
that coincide with $\R\times L$
away from a compact set, but are not Hamiltonian isotopic to $\R\times L$ even through
isotopies that are not compactly supported and ``slide'' the ends along themselves.
 Thus, we produce examples $L$ so that the set $\mor_{\cob^{d}_{0}}(L,L)$ is not trivial.  In our examples $L$ is a real Lagrangian in a toric symplectic manifold and we can also provide a lower bound on
the size of $\mor_{\cob^{d}_{0}}(L,L)$.
Finally, for the convenience of the reader, in \S\ref{subsubsec:flavors-Seidel}
we make explicit the equivalent descriptions of the Seidel morphism that appear 
in a variety of  sources in the literature as well as in the paper itself.

\

\noindent {\bf Acknowledgements.} We thank Jean-Fran\c{c}ois
Barraud and Paul Biran, for useful
discussions and Cl\'ement Hyvrier for sharing with us an early
version of his work \cite{Hy:circle}. The first author also thanks
Luis Haug and Yaron Ostrover for useful input.

% !TEX root = Serep.tex

\section{The ingredients in diagram (\ref{eq:main-diag}).}\label{sec:ingred}

\subsection{The categories in the diagram}\label{subsec:categ-diag}
\subsubsection{The Hamiltonian diffeomorphisms group.}
The Hamiltonian diffeomorphisms group of $M$ is denoted by $Ham(M)$
and $\pi_{1}(Ham(M))$ is its fundamental group, the base point is
the identity.  The group $\pi_{1}(Ham(M))$ is viewed here as a
category with one element and so that each element of the group is a
morphism from that element to itself. This category is monoidal in
an obvious way. The fundamental groupoid of $Ham(M)$, denoted by
$\Pi (Ham(M))$, is a category having as objects the Hamiltonian
diffeomorphisms of $M$ and the morphisms are homotopy classes of
paths in $Ham(M)$ relating two such diffeomorphisms.
\subsubsection{Quantum Homology.}
The category $QH(M)^{\ast}$ has a single object and its morphisms are the invertible elements
in the quantum homology of $M$. The composition is the quantum product.
\subsubsection{The cobordism category.}
This subsection contains a summary of the construction of
the  category $\cob^{d}_{0}(M)$  following \cite{Bi-Co:cob2}.

\

\noindent \underline{A. Monotonicity assumptions.} All families of
Lagrangian submanifolds in our constructions have to satisfy a
monotonicity condition in a uniform way as described below. Given a
Lagrangian submanifold $L \subset M$, let
$$\omega : \pi_{2}(M,L)\to \R \ , \ \mu:\pi_{2}(M,L)\to \Z$$
be the morphisms given respectively by integration of $\omega$ and
by the Maslov index. The Lagrangian $L$ is \emph{monotone} if there
exists a positive constant $\rho>0$ so that for all $\alpha\in
\pi_{2}(M,L)$ we have $\omega(\alpha)=\rho\mu(\alpha)$ and the
minimal Maslov number $N_{L}=\min\{\mu(\alpha) : \alpha\in
\pi_{2}(M,L) \ ,\ \mu(\alpha)>0\}$ satisfies $N_{L}\geq 2$. We work
at all times over $\mathbb{Z}_2$ as ground ring.

For a  closed, connected, monotone Lagrangian $L$ there is an associated
invariant $d_{L}\in \Z_{2}$ given
as the number (mod $2$) of $J$-holomorphic disks of Maslov
index $2$ going through a generic point $P\in L$ for $J$ a generic
almost complex structure that is compatible with $\omega$.

A family of Lagrangian submanifolds $L_{i}$, $i\in I$, is uniformly
monotone if each $L_{i}$ is monotone and the following condition is
satisfied: there exists $d\in K$ so that for all $i\in I$ we have
$d_{L_{i}}=d$ and there exists a positive real constant $\rho$ so
that the monotonicity constant of $L_{i}$ equals $\rho$ for all
$i\in I$. All the Lagrangians used in the paper will be assumed
monotone and, similarly, the Lagrangian families will be assumed
uniformly monotone. For $d\in \Z_{2}$ and $\rho > 0$, we let
$\mathcal{L}_{d}(M)$ be the family of closed, connected Lagrangian
submanifolds $L\subset M$ that are monotone with monotonicity
constant $\rho$ and with $d_{L}=d$.

\

\noindent\underline{ B. Cobordism.}
The plane $\mathbb{R}^2$  is endowed with the symplectic
structure $\omega_{\mathbb{R}^2} = dx \wedge dy$, $(x,y) \in
\mathbb{R}^2$.  The product $\widetilde{M}=\mathbb{R}^2 \times M$ is endowed with
the symplectic form $\omega_{\mathbb{R}^2} \oplus \omega$. We denote by
$\pi: \mathbb{R}^2 \times M \to \mathbb{R}^2$ the projection. For a
subset $V \subset \mathbb{R}^2 \times M$ and $S \subset \mathbb{R}^2$
we let $V|_{S} = V \cap \pi^{-1}(S)$.

\begin{dfn}\label{def:Lcobordism}
   Let $(L_{i})_{1\leq i\leq k_{-}}$ and $(L'_{j})_{1\leq j\leq
     k_{+}}$ be two families of closed Lagrangian submanifolds of
   $M$. We say that these two (ordered) families are Lagrangian
   cobordant, $(L_{i}) \simeq (L'_{j})$, if there exists a smooth
   compact cobordism $(V;\coprod_{i} L_{i}, \coprod_{j}L'_{j})$ and a
   Lagrangian embedding $V \subset ([0,1] \times \mathbb{R}) \times M$
   so that for some $\epsilon >0$ we have:
   \begin{equation} \label{eq:cob_ends}
      \begin{aligned}
         V|_{[0,\epsilon)\times \mathbb{R}} = & \coprod_{i}
         ([0, \epsilon) \times \{i\})  \times L_i \\
         V|_{(1-\epsilon, 1] \times \mathbb{R}} =
         & \coprod_{j} ( (1-\epsilon,1]\times \{j\}) \times L'_j~.~
      \end{aligned}
   \end{equation}
   The manifold $V$ is called a Lagrangian cobordism from the
   Lagrangian family $(L'_{j})$ to the family $(L_{i})$. We
   denote it by $V:(L'_{j}) \cobto (L_{i})$ or
   $(V; (L_{i}), (L'_{j}))$.
\end{dfn}
\begin{figure}[htbp]
   \begin{center}
      \epsfig{file=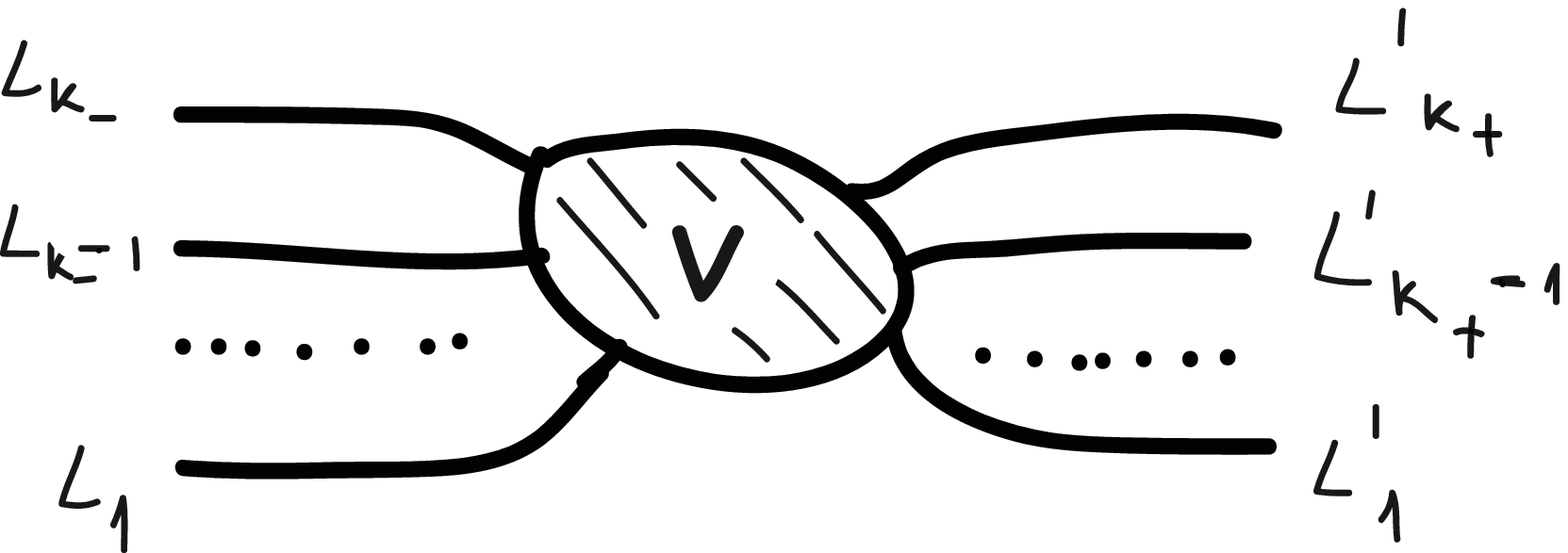, width=0.6\linewidth}
   \end{center}
   \caption{A cobordism $V:(L'_{j})\cobto (L_{i})$ projected
     on $\mathbb{R}^2$.}
\end{figure}

 A cobordism is called {\em monotone} if
$$V\subset ([0,1]\times \mathbb{R})\times M$$ is a
monotone Lagrangian submanifold.

We often view cobordisms as embedded in $\mathbb{R}^2 \times M$.
Given a cobordism $V\subset ([0,1] \times \mathbb{R}) \times M$ as
in Definition~\ref{def:Lcobordism}, we can extend trivially its
negative ends towards $-\infty$ and its positive ends to $+\infty$
thus getting a Lagrangian $\overline{V} \subset \mathbb{R}^2 \times
M$ and  we will in general not distinguish between $V$ and
$\overline{V}$.

\begin{rem} The notion of Lagrangian cobordism was introduced by Arnold \cite{Ar:cob-1,Ar:cob-2} in a slightly
less general form than above. Flexible aspects were discussed in early work of Eliashberg \cite{El:cob} and Audin \cite{Aud:calc-cob}. The first  indication that rigidity is also relevant to this notion of cobordism appeared in the work of Chekanov \cite{Chek:cob}. See \cite{Bi-Co:cob1} for more background material.
\end{rem}

Two Lagrangian cobordisms $V$, $V'\subset \C\times M$  are called {\em horizontally} isotopic if there is an ambient Hamiltonian isotopy $\psi_{t}$
(not necessarily compactly-supported) so that $\psi_{0}=\Id$, $\psi_{1}(V)=V'$ and
each $\psi_{t}$  slides along the ends of $V$ with bounded speed.
In particular, the ends of $V$ and $\psi_{t}(V)$ coincide away from a set $K\times M$ with $K\subset \C$ compact for all $t\in[0,1]$ (see \cite{Bi-Co:cob2} for a more explicit form of this definition).

\

\noindent \underline{C. The category $\mathcal{C}ob_{0}^{d}(M)$.}
Consider first the category
$\widetilde{{\mathcal{C}ob}}^{d}_{0}(M)$. Its objects are
families $$(L_{1}, L_{2},\ldots, L_{r})$$ with $r \geq 1$, $L_{i}\in
\mathcal{L}_{d}(M)$ so that additionally:
\begin{itemize}
\item[i.]
 for all $1\leq i\leq r$, if $L_{i}$ is non-void, then it is
 non-narrow (recall that a
monotone Lagrangian is non-narrow if its quantum homology $QH(L)$
with $\La=\Z_{2}[t,t^{-1}]$ coefficients does not vanish.
\item[ii.]  for all $1\leq i\leq r$ we have that the morphism
\begin{equation}\label{eq:Hlgy-vanishes}
    \pi_{1}(L_{i})\stackrel{i_{\ast}}{\longrightarrow}
   \pi_{1}(M)
   \end{equation}
   induced by  the inclusion $L_{i}\subset M$ vanishes.
\end{itemize}

\begin{rem}
In case the first Chern class $c_{1}$ and $\omega$ are proportional
as morphisms defined on $H_{2}(M;\Z)$ (and not only on
$\pi_{2}(M)$), then it is enough to assume in
(\ref{eq:Hlgy-vanishes}) that the image of $i_{\ast}$ is torsion.
\end{rem}

We denote by $\mathcal{L}^{\ast}_{d}(M)$ the Lagrangians in
$\mathcal{L}_{d}(M)$ that are non-narrow and additionally
satisfy~\eqref{eq:Hlgy-vanishes}.
Similarly, we denote by $\mathcal{L}_{d}(\C\times M)$
the Lagrangians in $\C\times M$ that
are uniformly monotone with the same $d_{V}=d$ and the same monotonicity
constant $\rho$ and we let $\mathcal{L}^{\ast}_{d}(\C\times M)$ be those Lagrangians
 $V\in \mathcal{L}_{d}(\C\times M)$ so that $\pi_{1}(V)\to \pi_{1}(\C\times M)$ is null.

\

We proceed to define the morphisms in
$\widetilde{\mathcal{C}ob}_{0}^{d}(M)$ . For any two horizontal
isotopy classes of cobordisms $[V]$ and $[U]$ with $V:(L'_{j})
\cobto (L_{i})$ (as in Definition~\ref{def:Lcobordism}) and
$U:(K'_{s})\cobto (K_{r})$, we define the sum $[V]+[U]$ to be the
horizontal isotopy class of a  cobordism $W:(L'_{j})+(K'_{s})\cobto
(L_{i})+(K_{r})$ so that $W=V \coprod \widetilde{U}$ with
$\widetilde{U}$ a suitable translation up the $y$-axis of a
cobordism horizontally isotopic to $U$ so that $\widetilde{U}$ is
disjoint from $V$.

The morphisms in $\widetilde{\mathcal{C}ob}^{d}_{0}(M)$ are now
defined as follows. A morphism $$[V] \in \mor \bigl( (L'_{j})_{1
\leq j \leq S}, (L_{i})_{1\leq i\leq T} \bigr)$$ is a horizontal
isotopy class that is written as a sum $[V]=[V_{1}] + \cdots +
[V_{S}]$ with each $V_{j}\in \mathcal{L}^{\ast}_{d}(\C\times M)$ a
cobordism from the Lagrangian family formed by the {\em single}
Lagrangian $L'_{j}$ and a subfamily $(L_{r(j)}, \ldots,
L_{r(j)+s(j)})$ of the $(L_{i})$'s, so that $r(j)+s(j)+1=r(j+1)$. In
other words, $V$ decomposes as a union of $V_{i}$'s each with a
single positive end but with possibly many negative ones - see
Figure \ref{fig:MorphCob}. We will often denote such a morphism by
$V: (L'_{j})\longrightarrow (L_{i})$.

The composition of morphisms is induced by concatenation followed by a
rescaling to reduce the ``width'' of the cobordism to the interval
$[0,1]$.

We consider here the void set as a Lagrangian of arbitrary
dimension. There is a natural equivalence relation among the objects of this category:
 it is induced by
the relations
\begin{equation}
(L, \emptyset) \sim (\emptyset,L) \sim (L).
\label{eq:eq-objects}
\end{equation}

There is also a corresponding equivalence relation for morphisms (we refer
again to \cite{Bi-Co:cob2} for details).
\begin{figure}[htbp]
   \begin{center}
      \epsfig{file=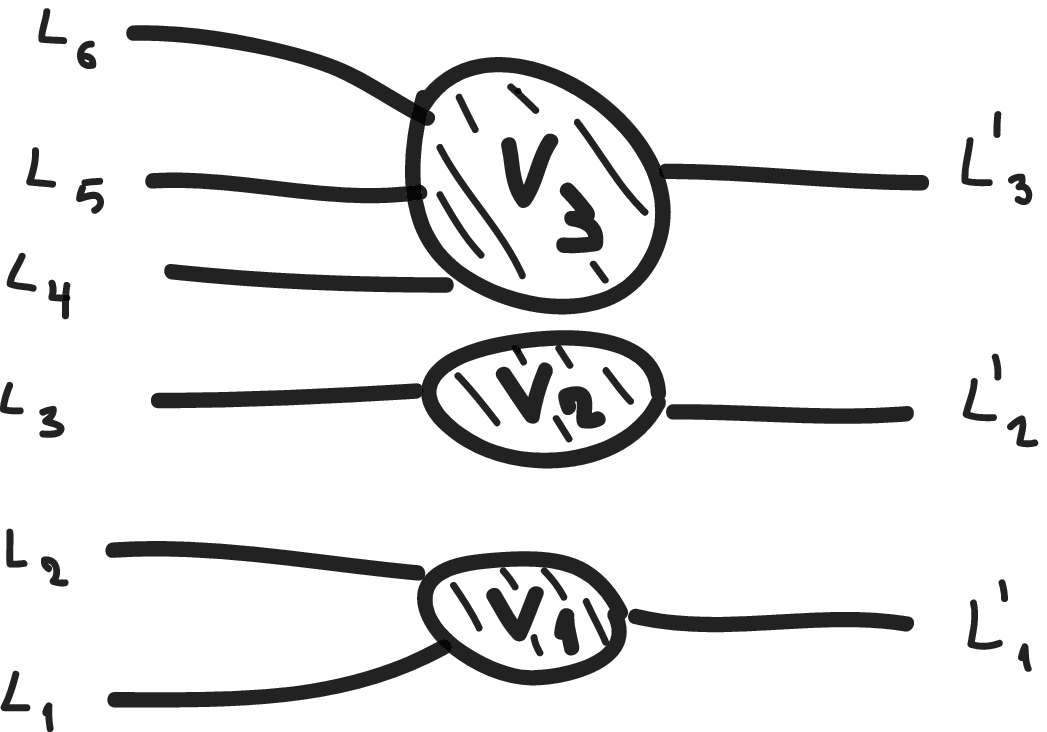, width=0.5\linewidth}
   \end{center}
   \caption{A morphism $V:(L'_{1},L'_{2},L'_{3}) \longrightarrow
     (L_{1}, \ldots, L_{6})$, $V=V_{1}+V_{2}+V_{3}$, projected to
     $\mathbb{R}^2$.}\label{fig:MorphCob}
\end{figure}

The category $\cob^{d}_{0}(M)$ is the quotient of $\widetilde{\cob}^{d}_{0}(M)$
with respect to the equivalence relations mentioned above (obviously, they are applied to
objects as well as to morphisms).

\begin{rem} The quotient $\cob^{d}_{0}(M)$ has a more complicated definition than
$\widetilde{\cob}^{d}_{0}(M)$ but it presents the advantage to avoid
redundancies related to the presence of multiple empty ends (these
are ends consisting of the empty Lagrangian). In particular, the
morphisms in $\cob^{d}_{0}(M)$ are represented by unique horizontal
isotopy classes of cobordisms $V$ (not necessarily connected) so
that the empty Lagrangian can appear as a positive (or negative) end
of $V$ only if there is a single positive end (or, respectively, a
single negative one). Moreover, $\cob^{d}_{0}(M)$ is strictly
monoidal.  \end{rem}

\subsubsection{The categories $D\fuk^{d}(M)$ and $T^{S}D\fuk^{d}(M)$}
We again recall a construction detailed in \cite{Bi-Co:cob2}.

\noindent\underline{A. The derived Fukaya category.} We review here
very briefly a few ingredients in the construction of this category.
This construction is presented in detail in
\cite{Se:book-fukaya-categ} and is summarized in the precise
monotone context we use here in \cite{Bi-Co:cob2}. In the same
monotone setting it has also appeared in the recent work of Sheridan
\cite{Sheridan1}.  We emphasize that we do not complete with respect
to idempotents.

The derived Fukaya category $D\fuk^{d}(M)$ is a triangulated
completion of the Donaldson category $\mathcal{D}on^{d}(M)$ . The
objects of $\mathcal{D}on^{d}(M)$ are Lagrangians $L\in
\mathcal{L}^{\ast}_{d}(M)$ and the morphisms are
$\mor_{\mathcal{D}on^{d}(M)}(L,L')=HF(L,L')$. Here $HF(L,L')$ is the
Floer homology of $L$ and $L'$. This is the homology of the Floer
complex $(CF(L,L'), d)$ that we consider here without grading and
over $\Z_{2}$. Assuming that $L$ and $L'$ are transverse, the vector
space $CF(L,L')$ has as basis the intersection points $L\cap L'$.
The differential $d$ counts $J$-holomorphic strips $u:\R\times
[0,1]\to M$ so that $u(\R\times \{0\})\subset L$ and $u(\R\times
\{1\})\subset L'$  and so that $\lim_{s\to-\infty}u(s,t)=x\in L\cap
L'$, $\lim_{s\to +\infty}u(s,t)=y\in L\cap L'$ ($J$ is a generic
almost complex structure on $M$). Monotonicity together with
condition (\ref{eq:Hlgy-vanishes}) are used to show the finiteness
of sums in the expression of $d$. Uniform monotonicity is again used
to show that $d^{2}=0$.  The main geometric step in the construction
of $D\fuk^{d}(M)$ is the construction of the Fukaya
$A_{\infty}$-category $\fuk^{d}(M)$. We will not review here the
definition of this $A_{\infty}$-category, we refer instead again to
\cite{Se:book-fukaya-categ}. Suffices to say that the objects in
this $A_{\infty}$-category are again Lagrangians in
$\mathcal{L}^{\ast}_{d}(M)$. At least naively, the morphisms
$\mor_{\fuk^{d}(M)}(L,L')=CF(L,L')$ (this is naive as morphisms have
to be defined even if $L$ and $L'$ are not transverse). The
structural maps are so that $\mu^{1}=d$ = the Floer differential and, for $k>1$,
$\mu^{k}$ is defined by counting $J$-holomorphic polygons with
$k+1$ sides. These polygons  have $k$ ``inputs'' asymptotic to
successive intersection points $x_{1}\in L_{1}\cap L_{2}$, $x_{2}\in
L_{2}\cap L_{3}$, $\ldots x_{k}\in L_{k}\cap L_{k+1}$  and one
``exit'' asymptotic to $y\in  L_{1}\cap L_{k+1}$ (this is again
naive for  the same reason as before: these operations have to be
defined for all families $L_{1},\ldots, L_{k+1}$ and  not only when
$L_{i}, L_{i+1}$, etc.,  are transverse). Again condition
(\ref{eq:Hlgy-vanishes}) and monotonicity  are used to show that the
sums appearing in the definition of the $\mu^{k}$'s are finite. One
then considers the category of modules over the Fukaya category
$$mod (\fuk^{d}(M)):=fun(\fuk^{d}(M), Ch^{opp})$$ where $Ch^{opp}$ is the opposite of the dg-category of chain complexes.
This $A_{\infty}$-category of modules is also actually a dg-category and
 is triangulated in the $A_{\infty}$-sense with
the triangles being inherited from the triangles in $Ch$ (where they
correspond simply to the usual cone-construction). At the same time, there is a Yoneda embedding
$\mathcal{Y}:\fuk^{d}(M)\to mod (\fuk^{d}(M))$, the functor
associated to an object $L\in \mathcal{L}^{\ast}_{d}(M)$ being
$CF(-,L)$.
 The derived Fukaya category $D\fuk^{d}(M)$ is the homology category associated to the triangulated
 completion of the image of the Yoneda embedding inside $mod(\fuk^{d}(M))$.

 \begin{rem}[The pullback functor]\label{rem:gen-non-sense1}
The following property of module categories will be useful later on.
Given two $A_\infty$ categories, $\mathcal{A}$ and $\mathcal{B}$,  the respective
module categories are related by an $A_\infty$ pullback functor
$*: fun(\mathcal{A}, \mathcal{B}) \to fun(mod (\mathcal{B}), mod (\mathcal{A}))$.
 This is
defined as follows (see also \cite[\S \textbf{(1k)})]{Se:book-fukaya-categ} as well 
as \cite[Appendix A]{Bi-Co:cob2}). 

A morphism in a  functor $A_{\infty}$ category, $ fun(\mathcal{C},\mathcal{C}')$, (such as, for instance, $mod(\mathcal{A})$, $mod(\mathcal{B})$, $fun(\mathcal{A},\mathcal{B})$) relating
two functors $U, U':\mathcal{C}\to\mathcal{C}'$  is, by definition, a pre-natural transformation $T$.  This consists of a couple 
 $(T^{0}, T')$ where $T^{0}$ is a collection of morphisms $T^{0}_{L}\in C(U(L),U'(L))$
 for each object $L$ of $\mathcal{C}$ and $T'$ is a collection
 of multilinear maps $$(T')^{d}: C(L_0, L_1) \otimes \cdots \otimes C(L_{d-1}, L_d) \to C (U(L_0), U' (L_{d}))$$ defined for each famliy $(L_{0},\ldots, L_{d})$ of objects
 of $\mathcal{C}$ and $d\geq 1$.  Such a collection is called an extended multilinear map; we denote
 by $C(-,-)$ the morphisms in the respective categories.

Let $F \in fun(\mathcal{A}, \mathcal{B})$. Define $F^*: mod (\mathcal{B}) \to mod (\mathcal{A})$ on objects by $F^{\ast}\phi= \phi \circ F$, where $\circ$ is the composition of $A_\infty$ functors (this is the obvious composition of 
mltilinear maps - see again \cite[Appendix A]{Bi-Co:cob2} for the notation).
By the definition of a functor between $A_{\infty}$ categories (this is again an extended multilinear map), 
we also need to define a sequence of multilinear maps 
$$(F^*)^d: C_{mod (\mathcal{B})}(\phi_0, \phi_1) \otimes \cdots \otimes C_{mod (\mathcal{B})}(\phi_{d-1}, \phi_d) \to C_{mod (\mathcal{A})}(F^* \phi_0, F^* \phi_{d}).$$
We set $(F^*)^d = 0$ when $d \geq 2$ and we are left with defining
\begin{align*}
(F^*)^1: C_{mod (\mathcal{B})}(\phi_0, \phi_1) & \to C_{mod (\mathcal{A})}(F^* \phi_0, F^* \phi_1)\\
\eta = (\eta^0, \eta') & \mapsto F^* \eta = (F^* \eta ^0, F^* \eta')~.~
\end{align*}
For every $L \in \text{Ob}(\mathcal{A})$ we put $(F^* \eta ^0)_L := \eta^0_{F(L)} \in C_{Ch^{opp}}(F^* \phi_0(L), F^* \phi_1(L))$ and  we define $F^* \eta'$ to be the 
extended multilinear map defined by $\eta' \circ F$. 
Consider now $T \in \mor_{fun(\mathcal{A}, \mathcal{B})}(F, G)$, $T = (T^0, T')$, a pre-natural transformation from $F$ to $G$.  We need to define a pre-natural transformation $T^*$ from $F^*$ to $G^*$.
We set $(T^*)' = 0$. Further, for each $\phi \in \text{Ob}(mod (\mathcal{B}))$ 
the element $(T^*)^0_\phi \in C_{mod (\mathcal{A})}(F^* \phi, G^* \phi)$ is a pre-natural transformation $(T^*)^0_\phi=\chi = (\chi^0, \chi')$ defined as follows.  For $L \in \text{Ob}(\mathcal{A})$, set $\chi^0_L := \phi(\eta^0_L) \in C_{Ch^{opp}}(\phi(F(L)), \phi(G(L)))$ and $\chi' := \phi \circ \star \circ (F, \eta, G)$ (see once more \cite[Appendix A]{Bi-Co:cob2} for the notation). Sketchily,
this composition  is $$\phi \circ \star \circ (F, \eta, G)(\cdot\cdot\cdot)=\sum \phi (F(\cdot\cdot\cdot),\ldots, F(\cdot\cdot\cdot), \eta(\cdot\cdot\cdot) , G(\cdot\cdot\cdot),\ldots, G(\cdot\cdot\cdot))~.~$$ 

Using the fact that module categories are dg-categories, direct computation shows that the 
pull-back so defined is indeed a functor of $A_{\infty}$ categories and, moreover,  if $T$ is a natural transformation (see \cite{Se:book-fukaya-categ} for the definition), then so are $\chi$ and $T^*$.  \end{rem}

\noindent \underline{B. The category $T^{S}D\fuk^{d}(M)$.}
Given any triangulated category $\mathcal{C}$, the associated category
$T^{S}\mathcal{C}$  has morphisms that reflect the various ways to
decompose objects in $\mathcal{C}$
 by iterated exact triangles.

A cone decomposition of length $k$ of an object $A\in \mathcal{C}$ is
a sequence of exact triangles:
$$T^{-1}X_{i}\stackrel{u_{i}}{\longrightarrow}Y_{i}
\stackrel{v_{i}}{\longrightarrow}Y_{i+1}\stackrel{w_{i}}{\longrightarrow}
X_{i}$$ with $1\leq i\leq k$, $Y_{k+1}=A$, $Y_{1}=0$. (Note that
$Y_2\cong X_1$.)  The sequence $(X_{1},\ldots, X_{k})$ is called the {\em linearization} of the 
cone decomposition. There is an obvious equivalence relation among cone decompositions.
The category $T^{S} \mathcal{C}$ called the {\em
  category of (stable) triangle (or cone) resolutions over}
$\mathcal{C}$ has as objects finite, ordered
families $(x_{1}, x_{2},\ldots, x_{k})$ of objects $x_{i}\in
\mathcal{O}b(\mathcal{C})$.

The morphisms in $T^{S} \mathcal{C}$ are simplest to describe when defined on
a family formed by a single object $x \in \mathcal{O}b(\mathcal{C})$
and target $(y_1, \ldots, y_q)$, $y_{i}\in \mathcal{O}b(\mathcal{C})$.
For this, consider triples $(\phi, a, \eta)$, where $a \in
\mathcal{O}b(\mathcal{C})$, $\phi: x \to T^s a$ is an isomorphism (in
$\mathcal{C}$) for some index $s$ and $\eta$ is a cone decomposition
of the object $a$ with linearization $(T^{s_1}y_1,
T^{s_{2}}y_{2},\ldots, T^{s_{q-1}}y_{q-1}, y_q)$ for some family of
indices $s_1, \ldots, s_{q-1}$.  A morphism $\Psi: x \longrightarrow
(y_1, \ldots, y_q)$ is an equivalence class of triples $(\phi,a,\eta)$
as before up to a natural equivalence relation (reflecting the equivalence of
cone decompositions).
We now define the morphisms between two general objects.
A morphism
$$\Phi\in\mor_{T^{S}\mathcal{C}}((x_{1},\ldots x_{m}),
(y_{1},\ldots, y_{n}))$$ is a sum $\Phi =\Psi_{1}\oplus \cdots \oplus
\Psi_{m}$ where $\Psi_{j}\in \mor_{T^{S}\mathcal{C}}(x_{j},
(y_{\alpha(j)},\ldots, y_{\alpha(j)+\nu(j)}))$, and $\alpha(1)=1$,
$\alpha(j+1) = \alpha(j) + \nu(j) + 1$, $\alpha(m) + \nu(m) = n$. The
sum $\oplus$ means here the obvious concatenation of morphisms. With
this definition this category is strict monoidal, the unit element
being given by the void family. We again refer to \cite{Bi-Co:cob2} for more
ample details.

The indexes $s_{i},s$ do not play any role in this paper as we work in an ungraded context.

There is a natural projection functor $\mathcal{P}:T^{S}\mathcal{C}\to \mathcal{C}$
that sends, on objects, $(x_{1},\ldots, x_{k})\to x_{k}$. On morphisms, we first define $\mathcal{P}$
for a simple morphism $\Psi:x\to (y_{1},\ldots, y_{q})$ represented by a triple $\Psi=(\phi,a,\eta)$ as
above. Denote by $T^{-1}x_{k}\to Z_{k}\to a \stackrel{\delta}{\longrightarrow} x_{k}$
the last exact triangle in the cone decomposition $\eta$.
In this case, $\mathcal{P}(\Psi)=\delta\circ\phi$. The value of $\mathcal{P}$ on a more general
 sum $\Phi=\Psi_{1}\oplus \Psi_{2}\ldots \oplus \Psi_{m}$ - again as above -  is
$\mathcal{P}(\Phi)=\mathcal{P}(\Psi_{m})$.

\begin{rem}\label{rem:general-non-sense0}
a. It is useful to notice that for two objects of $T^{S}\mathcal{C}$ that
coincide with single objects $a, a'\in \mathcal{O}b(\mathcal{C})$,
the morphisms from $a$ to $a'$ in $T^{S}\mathcal{C}$ are the {\em isomorphisms}
from $a$ to $a'$ in $\mathcal{C}$. In particular, not all the morphisms in $\mathcal{C}$
appear as morphisms in $T^{S}\mathcal{C}$.

b. We will need further the following simple remark. Assume that
$\mathcal{H}: \mathcal{C}\to\mathcal{C}$ is a
 functor that preserves the triangulated structure.
 It is then easy to see that  $\mathcal{H}$ admits a lift $\hat{\mathcal{H}}:T^{S}\mathcal{C}\to
 T{^S}\mathcal{C}$. Further, assuming that $H'$ is another such functor and that
  $\eta: \mathcal{H}\to \mathcal{H'}$ is a natural isomorphism, then
$\eta$ induces a natural transformation $\hat{\eta}:\hat{\mathcal{H}}\to \hat{\mathcal{H}'}$.
\end{rem}

\subsubsection{The endomorphism category}\label{section:general-non-sense2} For completeness, given a category
$\mathcal{X}$, we recall the (strict) monoidal structure on the endofunctor
category $\mathcal{E}nd(\mathcal{X})$; see also \cite{MacLane:categories}.  In brief, a category $C$ is strict monoidal if its objects as well as its morphisms can be multiplied via a bifunctor $C\otimes C\to C$. This multiplication is required to be associative and there has to be a unit object.
A functor is (strict) monoidal if it preserves the multiplication and the unit.
For non-strict monoidal categories the multiplication is only associative up to a natural isomorphism
and similarly for the axioms verified by the unit. 

On objects of $\mathcal{E}nd(\mathcal{X})$, multiplication is
simply the composition of functors.  On morphisms, given two natural
transformations $T_i \in \mor(F_i, G_i)$, $i=1,2$,  the product $T_1
\times T_2 \in \mor(F_2 \circ F_1, G_2 \circ G_1)$ is defined as
follows.  For any pair of objects $X, Y$ and any morphism $f: X \to
Y$, we have the following commutative diagrams:
$$
\xymatrix{
F_i(X) \ar[r]^{F_i(f)} \ar[d]_{T_i} & F_i(Y) \ar[d]^{T_i}\\
G_i(X) \ar[r]_{G_i(f)} & G_i(Y)
}
$$

We apply the functor $F_2$ to the diagram associated to $T_1$ and then use $T_2$ in the following fashion:
$$
\xymatrix{
F_2(F_1(X)) \ar[rr]^{F_2(F_1(f))} \ar[d]_{F_2(T_1)} & & F_2(F_1(Y)) \ar[d]^{F_2(T_1)}\\
F_2(G_1(X)) \ar[rr]^{F_2(G_1(f))} \ar[d]_{T_2} & & F_2(G_1(Y)) \ar[d]^{T_2}\\
G_2(G_1(X)) \ar[rr]^{G_2(G_1(f))} & & G_2(G_1(Y))
}
$$
We then take $T_2 \circ F_2(T_1)$ to be the desired natural transformation.  Simply put, we follow the natural transformations $F_2 \circ F_1 \stackrel{T_1}{\longrightarrow} F_2 \circ G_1\stackrel{ T_2}{\longrightarrow} G_2 \circ G_1$.
Clearly, we could as well have taken
$F_2 \circ F_1 \stackrel{T_2}{\longrightarrow} G_2 \circ F_1\stackrel{ T_1}{\longrightarrow} G_2 \circ G_1$ but because the square below commutes,
$$
\xymatrix{
F_2(F_1(X)) \ar[rr]^{F_2(T_1)} \ar[d]_{T_2} & & F_2(G_1(X)) \ar[d]^{T_2}\\
G_2(F_1(X)) \ar[rr]^{G_2(T_1)} & & G_2(G_1(X))
}
$$
this does not make any difference.

\subsection{The functors in the diagram}
\subsubsection{The functors in the top square.}\label{subsubsec:top-funct}
The inclusion $i:\pi_{1}(M)\to \Pi(Ham(M))$ is viewed here as an inclusion of categories and is
easily seen to be monoidal.

\

a. The functor $S$ is  Seidel's ``standard'' representation
\cite{Se:pi1}. One possible definition is as follows. Fix a loop of
Hamiltonian diffeomorphisms $\mathbf{g}=\{g_{t}\}_{t\in S^{1}}$.
Assume that $\mathbf{g}$ is induced by the Hamiltonian vector field
$X^{G}$ associated to a time dependent Hamiltonian $G:S^{1}\times
M\to \R$. Additionally, let $H:S^{1}\times M\to \R$ be another time
dependent Hamiltonian and let $H'(t,x)=G(t,g_{t}(x))+H(t,g_{t}(x))$.
Consider an almost complex structure (possibly time-dependent) $J$
that is compatible with $\omega$. Consider also the almost complex
structure $\tilde{J}=(g_{t})_{\ast}(J)$.  Assume that $J$ is so that
both $(H,J)$ and $(H',\tilde{J})$ are regular in the sense that the
Floer complexes $CF(H,J)$ and $CF(H',\tilde{J})$ are both defined.
Recall that there is a chain morphism $\psi :CF(H',\tilde{J})\to
CF(H, J)$ that is unique up to homotopy, induces an isomorphism in
homology and is associated to a generic homotopy $(H_{\tau},
J_{\tau}): (H',\tilde{J})\simeq (H, J)$, $\tau\in [0,1]$ of the
data. A $1$-periodic orbit $\gamma(t)$ of $H$ transforms into a
$1$-periodic orbit of $H'$ under the transformation $\gamma(t)\to
\gamma'(t)=g_{t}^{-1}(\gamma(t))$. Moreover, let $u:\R\times
S^{1}\to M$ be a solution to Floer's equation
$\partial_{s}u+J\partial_{t}u+\nabla H(t,u)=0$ and let
$v(s,t)=(g_{t})^{-1}(u(s,t))$. It is easy to see that $v$ verifies
the equation $\partial_{s}v+\tilde{J}\partial_{t}v+\nabla H'(t,v)=0$
(the gradient is taken in each equation with respect to the metric
associated to the respective almost complex structure).  As the
solutions of Floer's equation form the moduli spaces used to define
the differential in the Floer complex,  one concludes that  the map
 $\gamma \to \gamma'$ provides a chain level isomorphism (also called the naturality isomorphism)
  $N:CF(H,J)\to CF(H',\widetilde{J})$. We thus obtain that the morphism $N\circ \psi$
  is an endomorphism of $HF(H,J)\cong QH(M)$. It is not hard to see that this
  morphism is a module morphism over the quantum ring $QH(M)$ and that it is invertible. Thus it
  is completely determined by its value  $(N\circ \psi) ([M])\in QH(M)^{\ast}$ 
  on the unit  $[M]\in QH(M)$ of quantum homology. This specific element
  $(N\circ\psi) ([M])$ will be denoted by $S(\mathbf{g})$ and it depends only on the homotopy class of $\mathbf{g}$.
  The resulting application $[\mathbf{g}]\in \pi_{1}(Ham(M))\to S(\mathbf{g})\in QH(M)^{\ast}$ is a group
  morphism known as Seidel's representation \cite{Se:pi1}.

  \begin{rem}
  In Seidel's paper \cite{Se:pi1} the actual transformation used is rather $\gamma(t)\to g_{t}(\gamma(t))$. Thus, our class $S(\mathbf{g})$ is actually the inverse of the class there.
  \end{rem}

b. The functor $\widetilde{S}$.   The main part of the construction
appears in \cite{Se:book-fukaya-categ}. First, one constructs a
category $\fuk^{d}(M)^{free}$ on which $Ham(M)$ acts freely. This
category has as objects pairs $(L,g)$ with $g\in Ham(M)$ and $L\in
\mathcal{L}^{\ast}_{d}(M)$. One then picks for $g=\id$ all the
almost complex structures, Floer and perturbation data needed to
define the Fukaya category $\fuk^{d}(M)$.  With these fixed choices we now denote
the resulting $A_{\infty}$ category by $\fuk^{d}(M,0)$.  The objects of this category
are $(L, \id)$. This data is then transported by $g$: the
Lagrangians by $L\to g(L)$, but also the almost complex structures
as well as all the auxiliary data.  This defines Fukaya categories
$\fuk^{d}(M,g)$ with objects $(L,g)$. The multiplications in
$\fuk^{d}(M)^{free}$ are constructed so as to extend those on
$\{\fuk^{d}(M,g)\}, g\in Ham(M)$ in the sense that each one of
$\fuk^{d}(M,g)$ is a full and faithful subcategory of
$\fuk^{d}(M)^{free}$ and in such a way that the multiplications are
equivariant under the action of $Ham(M)$. This is possible because
the action of $Ham(M)$ is free on objects. To fix notation, if $h\in Ham(M)$
the action is $h(L,g)= (h(L),h\circ g)$. Thus, to each element
$g\in Ham(M)$ we can associate an obvious functor
$\widetilde{g}:\fuk^{d}(M)^{free}\to \fuk^{d}(M)^{free}$.  Given two
elements $g_{0}, g_{1}\in Ham(M)$ together with a path
$\mathbf{g}=\{g_{t}\}_{t\in [0,1]}$, $g_{t}\in Ham(M)$ that joins
them, there is  an associated natural transformation
$\xi_{\mathbf{g}}:\widetilde{g}_{0}\to \widetilde{g}_{1}$ that is
constructed in \cite{Se:book-fukaya-categ} (it appears there only
when $g_{0}=\id$ but the construction is the same in the general
case).  The definition of $\xi_{\mathbf{g}}$ is needed further in
the paper so we review it  shortly here. The natural transformation
$\xi_{ \mathbf{g}}$ is a collection of multi-linear maps defined for
any $k+1$ objects of $\fuk^{d}(M)^{free}$, $(L_{1}, h_1),\ldots (L_{k+1}, h_{k+1})$:
$$\xi_{\mathbf{g}}^{k}:CF((L_{1}, h_1), (L_{2}, h_2))\otimes\ldots CF((L_{k}, h_k), (L_{k+1}, h_{k+1}))\to CF(g_{0}(L_{1}, h_1), g_{1}(L_{k+1}, h_{k+1}))$$
together with some elements $\xi_{\mathbf{g}}^{0}\in CF(g_{0}(L, h), g_{1}(L, h))$ that are defined for each $(L,h)$. We summarize the construction of these maps next. It is not difficult to verify that, with the
definitions below, $\xi_{ \mathbf{g}}$ is indeed a natural transformation and not only a pre-natural one 
by the same argument as that given in \cite[\S \textbf{(10d)}]{Se:book-fukaya-categ}. 

We describe first the  maps $\xi_{\mathbf{g}}^{k}$ for $k\geq 1$. They are defined as follows.
Consider $(k+1)$-pointed stable disks $S'_{r}$ endowed with an interior marked point $z=z_{S'_{r}}$ ($r$ is here a parameter
moving inside the appropriate Deligne-Mumford-Stasheff associahedron).
This marked point is used to stabilize the disk in the source space (see also \cite[\S \textbf{(10d)}]{Se:book-fukaya-categ}) and to allow for a way to parametrize the choices of perturbation data compatible
with splitting and gluing.  The $k+1$ boundary  punctures are ordered clockwise. We denote by $C_{j}\subset
\partial S'_{r}$ the connected components of $\partial S'_{r}$ indexed
so that $C_{1}$ goes from the exit to the first entry, $C_{j}$ goes
from the $(j-1)$-th entry to the $j$-th, $1\leq j\leq k$, and $C_{k+1}$
goes from the $k$-th entry to the exit.  Up to reparametrization we may assume that the point $z_{S_{r}}=0$
and that the exit is identified with $(1,0)$.  Each of the other $k$ punctures, $z_{j}$, $1\leq j\leq k$
 corresponds to an entry and, in this writing,  is given by $z^{j}=e^{-2 i \pi  s_{j+1}}$ for a well defined
 $s_{j}\in [0,1]$.  We put $s_{1}=0$ and $s_{k+1}=1$ and we let $\mathbf{s}=\mathbf{s}_{S'_{r}}= \{0=s_{1}\leq s_{2}\leq \ldots \leq s_{k+1}\}$. The map $\xi_{\mathbf{g}}^{k}$ counts
perturbed $J$-holomorphic polygons, as in the definition of the
$A_{\infty}$- structure $\fuk^{d}(M)$, but subject to moving
boundary conditions controlled by $\mathbf{s}$:  if $z\in C_{j}$,
then $u(z)$ moves along $g_{s}(L_{j})$ for $s_{j}\leq s \leq
s_{j+1}$ so that, with the parametrization above, $u(e^{-2 i \pi
s})\in g_{s}(L_{j})$. The asymptotic conditions along the strip like
ends of $S_{r}$ are so that for the $j$-th input the asymptotic
limit belongs to $CF(g_{s_{j+1}}(L_{j}, h_j) , g_{s_{j+1}}(L_{j+1}, h_{j+1}))$ and
the output corresponds to an asymptotic limit in $CF(g_{0} (L_{1}, h_1),
g_{1} (L_{k+1}, h_{k+1}))$. Notice that  the complex $CF(g_{s_{j+1}}(L_{j}, h_j) ,
g_{s_{j+1}}(L_{j+1}, h_{j+1}))$ is the same as $CF((L_{j}, h_j), (L_{j+1}, h_{j+1}))$ because
of the construction of the category $\fuk^{d}(M)^{free}$.

The elements $\xi_{\mathbf{g}}^{0}\in CF(g_{0}(L, h),g_{1}(L, h))$ are given by counting curves with domain a stable disk  with a single boundary puncture - identified with $(1,0)$ - and  again together with an interior marked point identified with $0$. We will continue to denote a general such curve by $S'_{r}$. 
To have a better intuition of the definition of $\xi_{\mathbf{g}}^{0}$ assume for a moment that $h=id$
and that $g_{0}(L)$ and $g_{1}(L)$ are in general position (if this is not the case we use, as always,
Hamiltonian perturbations). Then $\xi_{\mathbf{g}}^{0}$ equals a sum: $\sum n_{i} x_{i}$
where the $x_{i}$'s are intersection points of $g_{0}(L)\cap g_{1}(L)$ and $n_{i}\in\Z$ counts the elements in the moduli space consisting of $J$-holomorphic curves $u:S'_{r}\to M$  
so that (assymptotically) $u(1,0)=x_{i}$ and, on the boundary, $u(e^{-2 i \pi t})\in g_{t}(L)$
(as always in this type of ``counting'' definitions, we restrict to the components of the relevant moduli space 
that are $0$-dimensional; monotonicity and Gromov compactness imply that the sum is finite).

Of course, for this construction to succeed one also needs to show that
the choices of interior marked points $z_{S'_{r}}$ together
with the additional choices specific to the definition of an $A_{\infty}$-category (strip-like ends, etc.) can all be made in a coherent way with respect to gluing and splitting in the Deligne-Mumford-Stasheff associahedron. This is achieved as in \cite{Se:book-fukaya-categ}. Moreover, it is easy to see that the fact that we are working in the monotone context and no longer in the exact one, as in \cite{Se:book-fukaya-categ}, can be dealt with by the methods  in \cite{Bi-Co:rigidity}.

We now consider the category of functors $fun(\fuk^{d}(M)^{free}, Ch^{opp})$
and we see  from Remark \ref{rem:gen-non-sense1}
that the functors $\widetilde{g}$ induce functors $$\hat{g}:fun(\fuk^{d}(M)^{free},Ch^{opp})
\to fun(\fuk^{d}(M)^{free},Ch^{opp})$$ given by composition
and the natural transformation $\xi_{\mathbf{g}}$ induces
 a natural transformation $\hat{\xi}_{\mathbf{g}}:\hat{g}_{0}\to \hat{g}_{1}$.
It is easily seen that the correspondence $g\to \hat{g}$,  $\mathbf{g}\to \hat{\xi}_{\mathbf{g}}$
provides an action of $\Pi (Ham(M))$ on $H fun(\fuk^{d}(M)^{free},Ch^{opp})$. In particular,
$[\hat{\xi}_{\mathbf{g}}]$ - the homology image of $\hat{\xi}_{\mathbf{g}}$ -
  only depends on the homotopy class (with fixed ends) of $\mathbf{g}$
and the natural transformations $[\hat{\xi}_{\mathbf{g}}]$ are natural isomorphisms.
The inclusions  $\fuk^{d}(M,g)\hookrightarrow \fuk^{d}(M)^{free}$ are all quasi-equivalences so that we can pull-back this action to $H fun(\fuk^{d}(M),Ch^{opp})$ as described in \cite{Se:book-fukaya-categ}. We will not change
the notation of the various functors, natural transformations, etc.,  after this pull-back.

The functors $\hat{g}_{0}$ and $\hat{g}_{1}$ preserve exact triangles and given
that $[\hat{\xi}_{\mathbf{g}}]$ is a natural isomorphism, we deduce from Remark \ref{rem:general-non-sense0} that the action of $\Pi (Ham(M))$ on $H fun(\fuk^{d}(M), Ch^{opp})$   induces an action $\widetilde{S}$ of $\Pi(Ham(M))$ on $T^{S}D\fuk^{d}(M)$.

\begin{rem}\label{rem:movingboundarydef} It will be useful in the following to have an alternative description of the homology classes $[\xi^{0}_{\mathbf{g}}]\in HF(g_{0}(L),g_{1}(L))$.  They can be described as follows: we consider a morphism
\begin{equation}\label{eq:moving-morph}
\phi_{\mathbf{g}}:CF(g_{0}(L), g_{0}(L))\to CF(g_{0}(L), g_{1}(L))
\end{equation}
defined by counting Floer strips
$u:\R\times [0,1]\to M$ with $u(\R\times \{0\})\subset g_{0}(L)$ and with $u(s\times \{1\}) \in g_{\psi (s)}(L)$
with $\psi : \R\to [0,1]$ an appropriate function that is increasing, null at $-\infty$ and $1$ at $+\infty$.
Then
\begin{equation}\label{eq:ident0}
[\xi^{0}_{\mathbf{g}}]=\phi_{\mathbf{g}}(PSS( [g_{0}(L)]))
\end{equation}
where $[g_{0}(L)]$ is the unit in $QH(g_{0}(L))$. 
The morphism $PSS: QH(g_{0}(L))\to HF(g_{0}(L), g_{0}(L))$ is the
PSS-isomorphism, $[g_{0}L]$ is the unit in $QH(g_{0}(L))$. 
The class $PSS( [g_{0}(L)])$ is itself the unit in
the ring $HF(g_{0}(L), g_{0}(L))$ and sometimes, in case no
confusion is possible, to shorten notation we will omit $PSS$ from
the notation of this homology class.

The proof of identity (\ref{eq:ident0}) is an exercise.  Indeed,
since $\xi^{0}_{\mathbf{g}}$ does not depend on the parametrization
of $\mathbf{g}$, we may choose the path $\mathbf{g}$ such that
 $g_s = g_0$ for $0 \leq s \leq 2/3$.   Fix now a Morse function $f:g_{0}(L)\to \R$ so
 that $f$ has a single maximum that we denote by $w$. In the pearl complex  \cite{Bi-Co:qrel-long}
  of $f$ this maximum represents $[g_{0}(L)]$.
 With this choice, for a generic metric on $g_{0}(L)$,
 the curves $u:S'_r \to M$ counted by $\xi^{0}_{\mathbf{g}}$ are in bijection
 with the pairs  $(u,\gamma)$ where  $\gamma:(-\infty, 0]\to g_{0}(L)$
 is a negative flow line of $f$ with origin at $w$ and so that $\gamma(0)=u(-1,0)\in u(S'_r)$. Indeed,
 by our choice of parametrization for $\mathbf{g}$, $u(-1, 0)\in g_0(L)$ and
thus, for a generic metric on $g_{0}(L)$, there is precisely one such flow line for each map $u$ (again, we restrict to $0$-dimensional moduli spaces for these counts so that there are actually a finite number 
of maps $u$ to worry about).   
The algebraic count of all  these pairs $(u, \gamma)$ is precisely the definition of the moving boundary PSS morphism (for instance see \cite{Bi-Co:cob1}), hence it is equal to $PSS([g_{0}(L)])$ in homology.
\end{rem}
c.  We now describe the action $\ast$. On objects, it associates to the unique object in $QH(M)^{\ast}$ the
identity functor of $T^{S}D\fuk^{d}(M)$. Given an element $\alpha\in QH(M)^{\ast}$ we need to explain
how we can associate to it a natural transformation $\hat{\eta}_{\alpha}$ from the identity to the identity in $T^{S}D\fuk^{d}(M)$.

By Remark \ref{rem:general-non-sense0},
to define such a transformation it is enough to produce a natural isomorphism of the identity
in $D\fuk^{d}(M)$.  The natural transformation $\hat{\eta}_{\alpha}$
is easy to describe on the Donaldson category. For this, fix $L\in \mathcal{L}_{d}^{\ast}(M)$.
Then $\hat{\eta}_{\alpha}$ associates to $L$ the
 homology class $\eta_{\alpha}(L)=\alpha \ast [L]\in QH(L)\cong HF(L,L)$. Here $\ast$ is the module action $QH(M)\otimes QH(L)\to QH(L)$. The properties
of this module action (as seen for instance in \cite{Bi-Co:qrel-long}) show that this definition induces a natural
transformation of the identity in the Donaldson category.
We now describe briefly the refinement of this definition at the level of
 the Fukaya category.

First, one defines a natural transformation
$\eta_{\alpha}: \id_{\fuk^{d}(M)}\to \id_{\fuk^{d}(M)}$ of $A_\infty$ functors by using moduli spaces of perturbed $J$-holomorphic polygons $S'_{r}$ with
one interior marked point $z_{S'_{r}}$ as were used at the point b
above.  The relevant moduli spaces are formed by curves $u:S'_{r}\to
M$ that satisfy the usual boundary conditions $u(C_{i})\subset
L_{i}$ and are so that if $c_{\alpha}$ is a cycle representing the
homology class $\alpha$ - for instance this could be a union of
unstable manifolds of an appropriate Morse function on $M$ - then
$u(z_{S'_{r}})\in c_{\alpha}$, thus yielding multilinear maps
$\eta_\alpha^k: CF(L_1, L_2) \otimes \cdots \otimes CF(L_k, L_{k+1}) \to CF(L_1, L_{k+1})$.  For each $L$, the morphism $\eta_L^0 \in CF(L, L)$ is defined by counting curves $u: S'_r \to (M, L)$, where $S'_r$ is a stable disk with one marked point on the boundary and one marked point in the interior.  The boundary point is mapped to an element of $CF(L, L)$ and the interior point $z_{S'_r}$ satisfies $u(z_{S_r}) \in c_\alpha$.

As a consequence of Remark \ref{rem:gen-non-sense1}, we obtain that
$\eta_{\alpha}$ induces a natural transformation
$\hat{\eta}_{\alpha}$ of the identity of
$fun(\fuk^{d}(M),Ch^{opp})$.  We then transport all this structure
to the homological category.  It is not difficult to show that the
resulting natural transformation $[\hat{\eta}_{\alpha}]$ does not
depend on the choices used in the construction.  It is then shown
that, additionally, for $\alpha,\beta\in QH(M)^{\ast}$, we have
$[\hat{\eta}_{\alpha\ast\beta}]=[\hat{\eta}_{\alpha}]\circ
[\hat{\eta}_{\beta}]$ . This identity is an obvious generalization
of the properties of the module action of $QH(M)$ on Lagrangian
quantum homology as explained for instance in \cite{Bi-Co:qrel-long}
\S 5.3. The proof uses moduli spaces of disks with boundary
punctures as before but with two internal marked points with the
constraint that, up to re-parametrization, one of the internal
marked points is at $(0,0)$, the output corresponds to a boundary
puncture at $(1,0)$ and the second internal marked point belongs to
the segment $(0,1)\times \{0\}\subset D^{2}$.

All the resulting natural transformations are natural isomorphisms
because the elements $\alpha$ are, by definition, invertible in the ring  $QH(M)$. It then follows that
$[\hat{\eta}_{\alpha}]$  induces a natural isomorphism $\bar{\eta}_{\alpha}$ of the identity on $D\fuk^{d}(M)$ and,
further, a natural transformation, also denoted by $\bar{\eta}_{\alpha}$, on $T^{S}D\fuk^{d}(M)$.
Moreover, $\bar{\eta}_{\alpha\ast\beta}=
\bar{\eta}_{\alpha}\circ\bar{\eta}_{\beta}$. This provides the action of $QH(M)^{\ast}$ on
 $T^{S}D\fuk^{d}(M)$.

\subsubsection{The functor $\widetilde{\mathcal{F}}$.}\label{sec:ftilde}
There is first an auxiliary functor $\mathcal{F}$ that will be
needed here and that we now recall from \cite{Bi-Co:cob1}. This
functor is a sort of linearization of $\widetilde{\mathcal{F}}$.
 It is defined on a simplified version of $\cob^{d}_{0}(M)$ that is denoted
by $S\cob^{d}_{0}(M)$. The objects of $S\cob^{d}_{0}(M)$ are single Lagrangians
$L\in \mathcal{L}^{\ast}_{d}(M)$ and the morphisms from $L$  to $L'$ are horizontal isotopy classes
of cobordisms $V:L\cobto (L_{1},\ldots, L_{k-1}, L')$, $V\in\mathcal{L}^{\ast}_{d}(\C\times M)$,
 where $k\in \N$ and $L_{i}\not=\emptyset$ if $k>0$ with the understanding that if $k=0$, then $V:L\cobto L'$.

The functor $$\mathcal{F}:S\cob^{d}_{0}(M)\to D\fuk^{d}(M)$$
is the identity on objects.  For a cobordism $V$ giving a morphism between $L$ and
$L'$ in $S\cob^{d}_{0}(M)$,  the morphism
 $\mathcal{F}(V)\in HF(L,L')=\mor_{D\fuk}(L,L')$
 is the image of the unity in $HF(L,L)$ (induced by the fundamental class of
$L$) through a morphism $\phi_{V}:HF(L,L)\to HF(L,L')$. The morphism
$\phi_{V}$ is given by counting Floer strips in $\mathbb{R}^2\times
M$ with boundary conditions along $V$ on one side and along
$\gamma\times L$ on the other side, $\mathcal{F}(V)=\phi_{V}([L])$.
Here $\gamma\subset \mathbb{R}^2$ and $V$ are as in
Figure~\ref{fig:MorphCob2}, with $L'=L_{k}$.
\begin{figure}[htbp]
   \begin{center}
      \epsfig{file=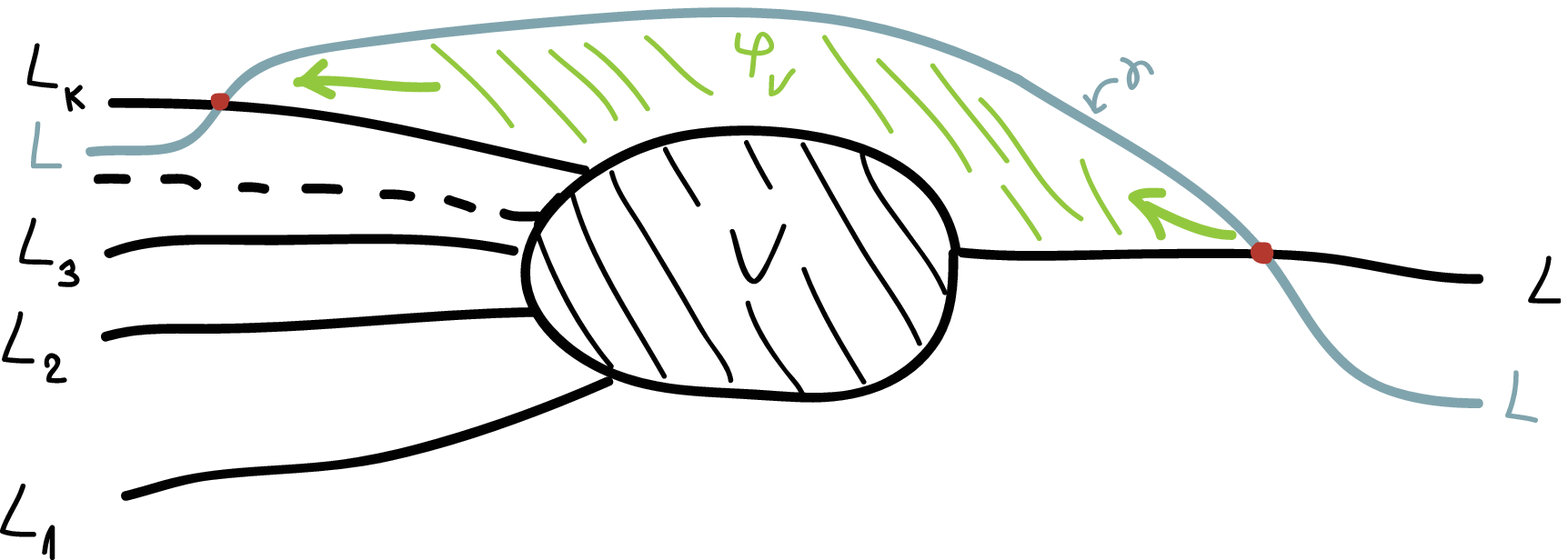, width=0.7\linewidth}
   \end{center}
   \caption{A cobordism $V\subset \mathbb{R}^2\times M$
     with a positive end $L$ and with $L'=L_{k}$ together with the projection of the
     $J$-holomorphic strips that define the morphism $\phi_{V}$.
     \label{fig:MorphCob2}}
\end{figure}
In case the cobordism $V$ is elementary - which means that $L$ is the single positive end of $V$ and
$L'$ is its single negative end - then $\phi_{V}$ as well as $\mathcal{F}(V)$
are isomorphisms.

\

The functor $$\widetilde{\mathcal{F}}:\cob^{d}_{0}(M)\to T^{S}D\fuk^{d}(M)$$
is constructed in \cite{Bi-Co:cob2}. Here we only need three of its properties:
\begin{itemize}
\item[i.] $\widetilde{\mathcal{F}}$ coincides with the identity on objects.
\item[ii.] $\widetilde{\mathcal{F}}$ is monoidal.
\item[iii.] For a morphism $V:L\to L'$, $V\in \mor_{\cob^{d}_{0}(M)}(L,L')$
 (that is, a morphism associated to an elementary cobordism)
 the value of $\widetilde{\mathcal{F}}$ coincides with the value on $V$ of the
  linearized functor $\mathcal{F}$:
 $$\widetilde{\mathcal{F}}(V)=\mathcal{F}(V)~.~$$ \end{itemize}
To make explicit the identity at iii recall that the value
$\widetilde{\mathcal{F}}(V)$ is by definition a triple $(\phi,
a,\eta)$ where $\phi: L\to a$ is an isomorphism in $D\fuk^{d}(M)$
and $\eta$ is a cone decomposition of $a$ in just one stage
$$T^{-1}L'\to 0\to a\stackrel{\delta}{\longrightarrow} L'~.~$$ In
other words  $\delta:a\to L'$ is an isomorphism that identifies
$L'$ and $a$. Thus  $\widetilde{\mathcal{F}}(V)$ reduces to giving
an isomorphism $\bar{\phi}:L\to L'$ in $D\fuk^{d}(M)$ which can then
be written as $\bar{\phi}=\delta\circ \phi$ (conversely, for any
such isomorphism, one can take $a=L'$, $\delta=\id$). As seen
above, because $V$ is elementary, $\mathcal{F}(V)$ is precisely such an
isomorphism $\in \mor_{D\fuk^{d}(M)}(L,L')=HF(L,L')$.

\subsubsection{The functor $\Sigma$.}
We first fix some notation. A Lagrangian cobordism $V: L \to (L_i)$ is written in coordinates as:
$$V \subset T^*[0,1] \times M\ , \ \  p \mapsto (t(p), y(p), \pi_M(p))~.~$$
We recall that the symplectic form on $T^{\ast}[0,1]\times M$ is $\bar{\omega}=\omega_{0}\oplus \omega$. The coordinates in $T^{\ast}[0,1]=[0,1]\times \R$ are $(t,y)$.

Given a path $\mathbf{g} = \{ g_t \}_{t \in [0,1]}, g_t \in Ham(M)$ generated by a time-dependent Hamiltonian function $G:[0,1]\times M\to \R$, we get a symplectomorphism
$$\Phi_G: T^*[0,1] \times M \to T^*[0,1] \times M, \ \  (t, y, p) \mapsto (t, y + G_t(g_t(p)), g_t(p)).$$
By restricting this map to the cobordism $V$, we get a new Lagrangian embedding
\begin{equation}\label{eq:isotopy1}\Phi_{G}: V \to T^*[0,1] \times M \ , \ \
 p \mapsto (t(p), y(p) + G_{t(p)}(g_{t(p)}(\pi_M(p))), g_{t(p)}(\pi_M(p)))~.~
 \end{equation}
We will assume here and below that, after a possible reparametrization,
the path $\mathbf{g}$ is constant close to each one of its ends. Moreover, the
Hamiltonian $G$ is supposed to be so that $G$ vanishes near the ends of the interval $[0,1]$
in the sense that  for some very small $\epsilon$  and any $m\in M$,
$G(t,m)=0$ for $t<\epsilon$ and $t>1-\epsilon$.
It follows that $V^{G}=\Phi_{G}(V): g_1(L) \to (g_0(L_i))$ is a cobordism. It will be called
the Lagrangian suspension of $V$ by the path $\mathbf{g}$. This agrees with standard terminology when $V$ is the trivial cobordism \cite{Po:hambook}.

\begin{prop}\label{prop:Susp-funct}
The Lagrangian suspension extends to an action of $\Pi (Ham(M))$ on $\cob^{d}_{0}(M)$.
\end{prop}
\begin{proof}

We start with the following geometric statement, which also shows that the suspension is independent of
the parametrization of the path $\mathbf{g}$.
\begin{lem}\label{prop:invariance}
Suppose $\mathbf{f}=(f_t)$ and $\mathbf{g}=(g_t)$ are paths in $Ham(M)$ that are induced, respectively,
by Hamiltonians $G$ and $G'$ as above. Further, assume that $\mathbf{f}$ and $\mathbf{g}$ are
 homotopic in $Ham(M)$ relative to their endpoints. Then the associated suspensions $V^{G}$
 and $V^{G'}$ are horizontally Hamiltonian isotopic for any cobordism $V$.
\end{lem}
\noindent {\em Proof of Lemma \ref{prop:invariance}.} Let $h:[0,1]\times [0,1]\to Ham(M)$ denote a homotopy so that $h(0,t) = f_t$ and $h(1,t) = g_t$
(the homotopy parameter is $s$). Up to reparametrization we may assume that the paths
of Hamiltonian diffeomorphisms
$h_{s}(t)$, $t\in [0,1]$ ($s$ fixed and $t$ varies) and $h_{t}(s)$, $s\in [0,1]$ ($t$ fixed and $s$ varies)
are constant near their ends for all $s,t\in [0,1]$.
We now consider two families of Hamiltonian vector fields induced by  $h$:
$$X_{s,t}(h(s,t)(m))= \frac{\partial}{\partial t} (h(s,t)(m)) , \quad  Y_{s,t}(h(s,t)(m))= \frac{\partial}{\partial s}
(h(s,t)(m))~.~$$
In particular, $X_{0,t}$ is the Hamiltonian vector field of $G$ and $X_{1,t}$ is the
Hamiltonian vector field of $G'$.

Let  $H,F: [0,1]\times [0,1]\times M\to \R$ be associated one parametric families of Hamiltonians so that
$\omega (- ,X_{s,t}) =d H_{s,t}(-)$ and
$\omega (-, Y_{s,t})=d F_{s,t}(-)$. The notation here is  $H_{s,t}(m)=H(s,t,m)$ (and similarly for $F$)
and we are taking exterior derivative in $M$ and keeping $s,t$ fixed.  We also assume that $H_{0}=G$ and $H_{1}=G'$ (this is not restrictive
because both $G$ and $G'$ vanish near their ends). Moreover, for each fixed $s$, the Hamiltonian $H(s,-,-)$
vanishes close to the ends of the interval $[0,1]$ and the same is true for each fixed $t$ for
the Hamiltonian $F(-,t,-)$.
Consider the following map:
\begin{align}\label{eq:glob-iso}
\Psi: [0,1]\times T^*[0,1] \times M  \to T^*[0,1] \times M \\ \nonumber  (s,t,y,m) \mapsto (t, y + H(s,t, h(s,t)(m)), h(s,t)(m))
\end{align}

It is easy to check that each map $\Psi_{s}$, obtained from $\Psi$ by keeping $s$ fixed,
 is symplectic.  In particular, 
  the restriction of $\Psi$ to any cobordism $V$ is a Lagrangian
 isotopy from $V^{G}$ to $V^{G'}$.  Moreover, given that $H(s,-,-)$ is vanishing near the ends
 of $[0,1]$ for each $s\in [0,1]$, to verify that $\Psi|_{V}$ is a horizontal isotopy
 it only remains to prove that $\Psi$ is Hamiltonian.
 Define for each $s\in [0,1]$,
$$\alpha_s (-)= \bar{\omega}(\Psi_{\ast}(\frac{\partial}{\partial s}), -)\in\Omega^{1}(T^{\ast}[0,1]\times M)~.~$$
Thus, we need to show that $\alpha_s$ is exact for each $s$.
Let $$\bar{F} :[0,1]\times T^*[0,1] \times M \to \mathbb{R}\ ; \ (s,t,y,m) \mapsto F(s,t,m)~.~$$
Denote by $D$ the exterior derivative on $T^{\ast}[0,1]\times M$.
We want to notice that $D\bar{F}_{s}=-\alpha_{s}$. For this we compute
$$\Psi_{\ast}(\frac{\partial}{\partial s})= (0, \frac{\partial H}{\partial s} +dH (\frac{\partial h}{\partial s}), \frac{\partial h}{\partial s})=(0,\frac{\partial H}{\partial s}+ dH(Y), Y)= (0,\frac{\partial F}{\partial t}, Y)$$
the last equality coming from the identity $\frac{\partial H}{\partial s}+\{H,F\}=
\frac{\partial F}{\partial t}$ (our convention for the Poisson bracket being $\{A,B\}=\omega(X^{B},X^{A})$).
We immediately deduce $\alpha_{s}(\frac{\partial}{\partial y})=D\bar {F}(\frac{\partial}{\partial y})=0$.
Further, $\alpha_{s}(\frac{\partial}{\partial t})=-\frac{\partial F_{s}}{\partial t}= -D\bar{F} (\frac{\partial}{\partial t})$. Finally, for $\xi\in TM$, we have
$\alpha_{s}(\xi)=-dF_{s}(\xi)=-D\bar{F}(\xi)$ which proves the claim.
\qed

It follows from the Lemma that the horizontal isotopy class of $V^{G}$ only depends on the
horizontal isotopy class of $V$ and on the homotopy class (with fixed endpoints) of $\mathbf{g}$.
We will denote this horizontal isotopy class by $[V]^{\mathbf{g}}$, in other words,
$[V]^{\mathbf{g}}=[V^{G}]$. By a slight abuse of notation we will generally
 denote by $V^{\mathbf{g}}$ the cobordism $V^{G}$ in case the choice of the Hamiltonian $G$ is not
 significant.

\

We  now proceed to define a monoidal functor
$$\Sigma :\Pi (Ham(M)) \to  \en(\cob^{d}_{0}(M)).$$

Let $g \in Ob(\Pi (Ham(M)))$ and set
\begin{align*}
\Sigma(g): \cob^{d}_{0}(M) & \longrightarrow \cob^{d}_{0}(M)\\
\left\{
\begin{array}{c}
L  \in Ob\\
\text{[} V] \in \mor
\end{array}
\right\}
& \longrightarrow
\left\{
\begin{array}{c}
g(L) \\
\text{[}g(V)]:= [(\id  \times g)(V)]
\end{array}
\right\}
\end{align*}

This is well defined and a functor by Lemma \ref{prop:invariance}.
Given a path of Hamiltonian diffeomorphisms, now viewed as a
morphism $\mathbf{g}: g_0 \to g_1$ in $\Pi (Ham(M))$ we need to
define a natural transformation $\Sigma(\mathbf{g})$ between  the
two functors $\Sigma(g_0)$ and $\Sigma(g_1)$. For each Lagrangian
$L\in \mathcal{O}b(\mathcal{C}ob^{d}_{0}(M))$ we define
$$\Sigma(\mathbf{g})(L)=[[0,1] \times L]^{\mathbf{g}^{-1}}~.~$$
 Here $\mathbf{g}^{-1}$
is the path $g_{1-t}$ (the term $1-t$ is due to the fact that morphisms in $\cob$ go from right to left).
For further use we will put more generally $\Sigma(\mathbf{g})(V)=[V]^{\mathbf{g}^{-1}}$.

Given a  cobordism $V: L \to (L_i)$, we need to show that the following
diagram commutes:

\begin{eqnarray}\label{eq:diag-sigma}
\xymatrix{
g_0(L) \ar[rr]^{g_0(V)} \ar[dd]_{\Sigma(\mathbf{g})(L)} \ar@{-->}[ddrr]^{\Sigma(\mathbf{g})(V)}& & (g_0(L_i)) \ar[dd]^{\coprod_i \Sigma(\mathbf{g})(L_i)}\\
& & \\
g_1(L) \ar[rr]_{g_1(V)} & & (g_1(L_i))
}
\end{eqnarray}
Consider two functions $a_{j}(t) : [0,1] \to [0,1]$, $j=0,1$ that are smooth,
 surjective and so that  $a_0(t) = 1$ for $\frac{1}{2} \leq t \leq 1$ and
 $a_1(t) = 0$ for $0 \leq t \leq \frac{1}{2}$.
 We use these functions to reparametrize the path $\mathbf{g}$. We put
 $g^{j}_{t}=g_{a_{j}(t)}$ and $\mathbf{g}^{j}=(g^{j}_{t})$.
 We also need two cobordisms $V_{0}$ and $V_{1}$ horizontally isotopic to $V$ and so that
 $V_{0}|_{[0,\frac{1}{2}]\times \R\times M}=\cup_{i}[0,\frac{1}{2}]\times \{i\}\times L_{i}$
 and $V_{1}|_{[\frac{1}{2},1]\times \R\times M}= [\frac{1}{2},1]\times \{1\}\times L$.
 By Lemma \ref{prop:invariance}
 we have $\Sigma(\mathbf{g})(V)=\Sigma({\mathbf{g}^{0}})(V_0)=\Sigma({\mathbf{g}^{1}})(V_{1})$.
 The commutativity of the diagram now follows by noticing that
 $\Sigma({\mathbf{g}^{0}})(V_0)$ represents the composition
 $(\coprod \Sigma(\mathbf{g})(L_{j}))\circ g_{0}(V)$ and
  $\Sigma({\mathbf{g}^{1}})(V_1)$
 represents the composition $g_{1}(V)\circ \Sigma(\mathbf{g})(L)$.

\

Next, we show that $\Sigma$ is a functor.  Given $\mathbf{f}, \mathbf{g}$ two morphisms in $\Pi(Ham(M))$ such that $f_1 = g_0$, we have the following composition of natural transformations:
$$
\xymatrix{ f_0(L) \ar[rr]^{\Sigma(\mathbf{f})(L)} & & f_1(L)
\ar[rr]^{\Sigma(\mathbf{g})(L)} & & g_1(L) } ~.~$$ Clearly,
$\Sigma(\mathbf{g})(L)\circ \Sigma(\mathbf{f})(L)=[[0,1] \times
L]^{(\mathbf{f}\#\mathbf{g})^{-1}}$ where $\mathbf{f}\#\mathbf{g}$
is the concatenation of the paths $\mathbf{f}$ and $\mathbf{g}$.
Given that in $\Pi(Ham(M))$ the composition of morphisms is so that
$\mathbf{f}\#\mathbf{g}=\mathbf{g}\circ \mathbf{f}$, we also have
$\Sigma (\mathbf{g}\circ \mathbf{f})(L)=[[0,1] \times
L]^{(\mathbf{f}\#\mathbf{g})^{-1}}$.

\

Finally, we need to show that $\Sigma$ is monoidal.  Let $\mathbf{f}, \mathbf{g}$ be two morphisms in $\Pi(Ham(M))$.  By the definition of the multiplication rule for natural transformations in \S \ref{section:general-non-sense2}, $\Sigma$ is monoidal if the following diagram of cobordisms commutes:
$$
\xymatrix{
g_0 f_0(L) \ar[drr]_{\Sigma(\mathbf{g} \cdot \mathbf{f})(L)} \ar[rr]^{g_0(\Sigma(\mathbf{f})(L))} & & g_0 f_1(L) \ar[d]^{\Sigma(\mathbf{g})(f_1(L))}\\
& & g_1 f_1(L)
}
$$
where $\mathbf{g} \cdot \mathbf{f}$ denotes the path $g_t  f_t\in Ham(M)$.
Notice that $\mathbf{g}\cdot \mathbf{f}$ is homotopic,
with fixed end-points, to the composition $(\mathbf{g}\cdot f_{1})\circ (g_{0}\cdot \mathbf{f})$
so that the commutativity above follows from Lemma \ref{prop:invariance} and this concludes the proof.
\end{proof}

% !TEX root = Serep.tex

\section{Commutativity of diagrams (\ref{eq:main-diag2}) and (\ref{eq:main-diag}).}
\label{sec:commut}
\subsection{ Proof of Theorem \ref{thm:main2}} \label{subsec:bottom}

As mentioned in the introduction, the commutativity of diagram
(\ref{eq:main-diag2}) is equivalent to the commutativity of the
diagram below:
\begin{eqnarray}\label{eq:comm3}
\begin{aligned}
   \xymatrix{
   \Pi (Ham(M))\ar[d]_{\Sigma}\ar[rrr]^{\widetilde{S}} & & &\en (T^{S}D\fuk^{d}(M))\ar[d]^{\widetilde{\mathcal{F}}^{\ast}}\\
    \en(\cob^{d}_{0}(M))\ar[rrr]_{\widetilde{\mathcal{F}}_{\ast}}& & & fun (\cob^{d}_{0}(M), T^{S}D\fuk^{d}(M))}
\end{aligned}
\end{eqnarray}
On objects this commutes in the sense that for any $g\in Ham(M)$ the following
diagram of functors commutes:
\begin{eqnarray}\nonumber
\begin{aligned}
   \xymatrix{
   \mathcal{C}ob^{d}_{0}(M) \ar[rrr]^{\widetilde{\mathcal{F}}}\ar[d]_{\Sigma (g)}& & &T^{S}D\fuk^{d}(M)\ar[d]^{\widetilde{S}(g)}\\
    \mathcal{C}ob^{d}_{0}(M)\ar[rrr]^{\widetilde{\mathcal{F}}}& & &  T^{S}D\fuk^{d}(M)}
\end{aligned}
\end{eqnarray}
This is immediate because $g$ acts on both $\mathcal{C}ob^{d}_{0}(M)$ and on $T^{S}D\fuk^{d}(M)$
by ``translating'' all the data by the Hamiltonian diffeomorphism $g$.

To prove the commutativity of (\ref{eq:comm3}) it remains to show that for any path $\mathbf{g}=(g_{t})_{t\in [0,1]}$ in $Ham(M)$ the natural transformations $\widetilde{S}(\mathbf{g})$ and $\Sigma(\mathbf{g})$ are related
by
\begin{equation}\label{eq:compo}
\widetilde{S}(\mathbf{g})\circ \widetilde{\mathcal{F}}=\widetilde{\mathcal{F}}\circ \Sigma(\mathbf{g})~.~
\end{equation}
Given that $\widetilde{\mathcal{F}}$ is monoidal this is sufficient to verify for single Lagrangians
$L\in \mathcal{O}b(\mathcal{C}ob^{d}_{0}(M))$. Recall that $\mathcal{\widetilde{F}}$ acts as the identity on
objects and that on elementary cobordisms (those that have a single positive and a single negative end)
it coincides with the linearization $\mathcal{F}$ of $\widetilde{\mathcal{F}}$.
It follows  that the proof reduces to showing  that
for all Lagrangians $L$ as before we have:
\begin{equation}\label{eq:compo2}
\widetilde{S}(\mathbf{g})(L)=\mathcal{F}(\Sigma(\mathbf{g})(L))~.~
\end{equation}

To shorten notation we now put
$$V=\Sigma(\mathbf{g})(L)=([0,1]\times L)^{\mathbf{g}^{-1}}~.~$$
Recall that $\mathcal{F}(V)\in HF(g_{0}(L),
g_{1}(L))=\mor_{D\fuk}(g_{0}(L),g_{1}(L))$ is defined as
$\mathcal{F}(V)=\phi_{V}([g_{0}(L)])$ where $[g_{0}(L)]\in
HF(g_{0}(L),g_{0}(L))$ is the unit. The morphism
$$\phi_{V}:CF(g_{0}(L),g_{0}(L))\to CF(g_{0}(L),g_{1}(L))$$ was
recalled in \S\ref{sec:ftilde} (see also \cite{Bi-Co:cob1}).

As discussed in Remark \ref{rem:movingboundarydef} -
 the term $\widetilde{S}(\mathbf{g})(L)$ coincides with the homology class $[\xi^{0}_{\mathbf{g}}]=\phi_{\mathbf{g}}([g_{0}(L)])$ where
$$\phi_{\mathbf{g}}:CF(g_{0}(L),g_{0}(L))\to CF(g_{0}(L), g_{1}(L))$$
is the moving boundary morphism (\ref{eq:moving-morph}).
Thus, to finish the proof it is enough to show the following result that
first appeared in \cite{Cha:thesis}.

\begin{lem}\label{lem:moving}
With the notation above
$\phi_{V}$ and  $\phi_{\mathbf{g}}$ induce the same morphism in homology.
\end{lem}
\noindent {\em Proof of Lemma \ref{lem:moving}.}
We will prove that for adequate choices of Floer data and other auxiliary data, the morphisms $\phi_{V}$ and $\phi_{\mathbf{g}}$ are chain homotopic.

Consider a one parameter family of paths $h(s,t):= g_{s(1-t)} \in
Ham(M)$, $s\in [0,1]$.
This gives a homotopy  between the constant path $\mathbf{g}_{0}=(g_0)$ and the path $\mathbf{g}$.  Obviously, the $1$ endpoint of this
homotopy is not fixed.  We now follow the general scheme in the proof of  Lemma \ref{prop:invariance} to define a global Hamiltonian
isotopy $\Psi: [0,1]\times T^{\ast}[0,1]\times M \to T^{\ast}[0,1]\times M$, $$\Psi(s,t,y,m) = (t, y + H(s,t, h(s,t)(m)), h(s,t)(m))$$  as in (\ref{eq:glob-iso}).
After an appropriate reparametrization of $h(s,t)$ we can again assume that $H(s,-,-):[0,1]\times M\to \R$ vanishes for $t$ sufficiently close to the ends of the interval $[0,1]$. The same argument as in the proof of
Lemma \ref{prop:invariance} shows that $\Psi$ is indeed Hamiltonian.
The isotopy $\Psi=\{\Psi_{s}\}_{s \in [0,1]}$
verifies:
\begin{itemize}
    \item[-] $\Psi_0 = \id \times g_0$;
    \item[-] $\Psi_1([0,1]\times L)  = \Sigma(\mathbf{g})(L)$;
    \item[-] $\Psi_{s}|_{\{1\} \times \{ 0 \} \times M} = g_0$;
    \item[-] $\Psi_s |_{\{0\} \times \{ 0 \} \times M} = g_s$.
\end{itemize}
In particular, $\Psi|_{[0,1]\times L}$ is not horizontal. We now
denote $V_{s}=\Psi_{s}([0,1]\times L)$ so that $V_{0}=[0,1]\times
g_{0}(L)$, $V_{1}=V=\Sigma(\mathbf{g})(L)$. We also extend all these
cobordisms trivially to $\C\times M$. Recall that $\pi :\C\times
M\to \C$ is the projection. The set $\pi(\bigcup_{s \in [0,1]} V_s)$
is contained in a compact $K\subset T^{\ast}[0,1]$ so that there
exists a curve $\gamma$ as in the picture below, with $\gamma
\cap\R\times \{0\}= \{(-1,0)\}\cup \{2,0)\}=\gamma \cap \pi(V_{s})$
for all $s\in [0,1]$,  and so that $\gamma$ intersects $\R\times
\{0\}$ transversally. The two points $Q=(-1,0)$ and $P=(2,0)$ will
be referred to as the ``bottlenecks'' (see \cite{Bi-Co:cob2}).

\begin{figure}[ht]\label{fig:intertime01}
    \centering
    \subfloat[Intersection of $V'=\gamma \times g_0(L)$ with $V_{0}=\R \times g_{0}( L)$]{\includegraphics[width=8cm]{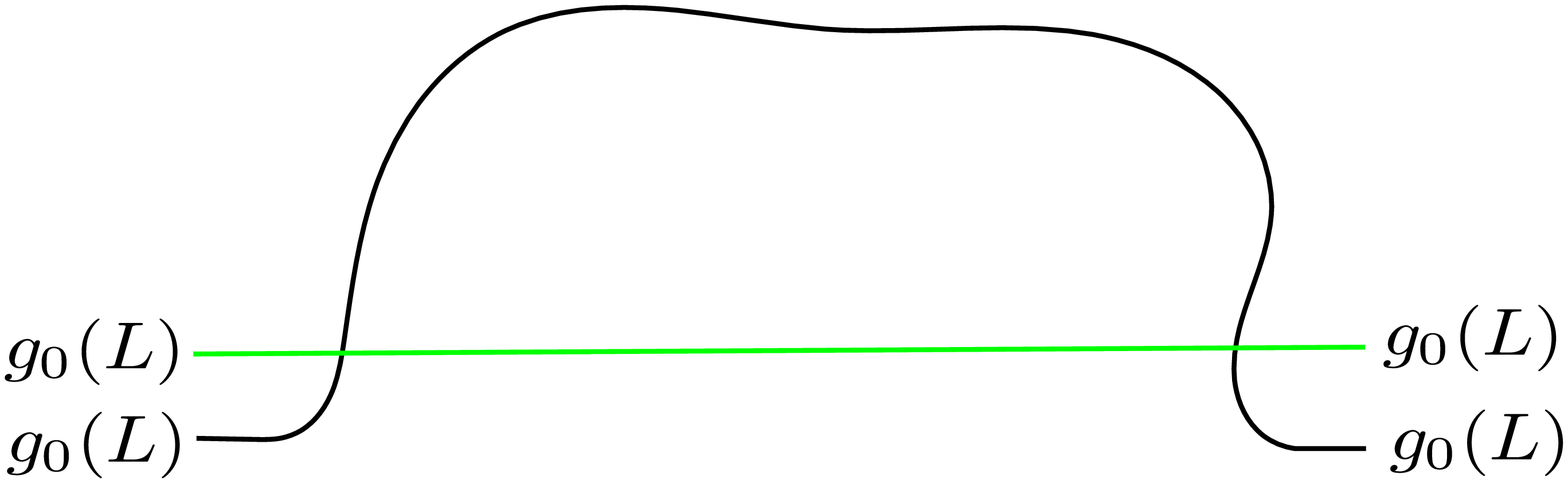}} \quad
    \subfloat[Intersection of $V'$ with $V_{1}=V=\Sigma(\mathbf{g})(L)$]{\includegraphics[width=8cm]{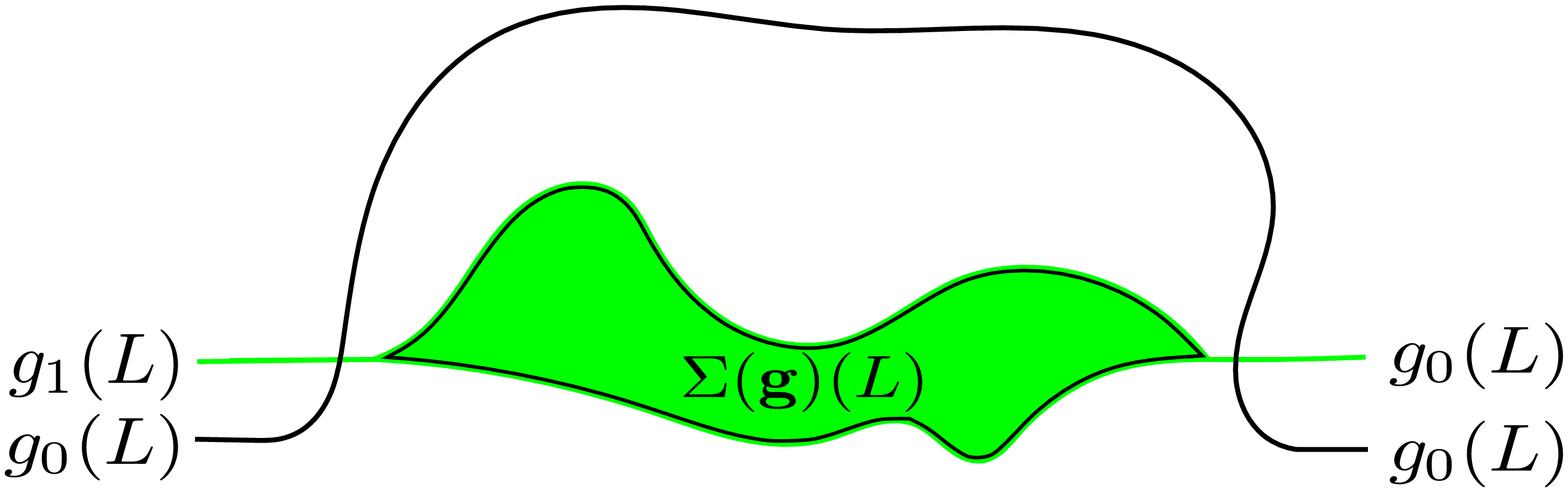}}
\label{fig:intertime01}\end{figure}

We let $V'=\gamma\times g_{0}(L)$ and consider the Floer complexes
$$\mathcal{C}_{1}=(CF(V', V_{0}), D_{1})\ \mathrm{and}\ \mathcal{C}_{2}=(CF(V', V), D_{2})~.~$$
The complex $\mathcal{C}_{2}$ is defined as in \cite{Bi-Co:cob1}
and \cite{Bi-Co:cob2} by first choosing Floer data for the pairs $(g_{0}(L), g_{0}(L))$  and $(g_{0}(L), g_{1}(L))$  in $M$. This data is used in the fibers of $\pi$ over the bottlenecks   and is then
 extended to the pairs $(V',V_{0})$ and $(V',V_{1})$, away from the bottlenecks,  by using an almost complex structure on $\C\times M$ so that the projection $\pi:\C\times M\to \C$ is holomorphic outside of $K$. For $\mathcal{C}_{1}$ we use the same procedure but
additionally it is easy to see
that one can work with an almost complex structure on $M\times \C$
so that $\pi$ is globally holomorphic. The construction here is in fact simpler than that in \cite{Bi-Co:cob1} because all the pairs $(V',V_{s})$
are cylindrically distinct at infinity.
By the same arguments as in \cite{Bi-Co:cob1} we deduce
 that the form of the two differentials appearing here is:
$$
D_1 =
\begin{pmatrix}
d_{0} & 0\\
1 & d_{0}
\end{pmatrix}\ \  ,
\qquad D_2 =
\begin{pmatrix}
d_{0} & 0\\
\phi_V & d_{1}
\end{pmatrix},
$$
where $d_{0}$ is the differential in $CF(g_{0}(L), g_{0}(L))$ and
$d_{1}$ is the differential in $CF(g_{0}(L),g_{1}(L))$.

We now define a chain morphism
$$\eta_{\Psi}:\mathcal{C}_{1}\to \mathcal{C}_{2}$$
by counting finite energy Floer trajectories $u(s,t):\R\times [0,1]\to \C\times M$ whose planar projection starts in $P$ and ends in $Q$ and so that
$u$ is subject to moving boundary conditions:
$$u(s,0) \in V_{0}, \qquad u(s,1) \in V_{\alpha(s)}$$
where $\alpha:\R\to [0,1]$ is a  smooth increasing  function  equal to $0$ for $s$ sufficiently small and equal to $1$ for $s$ sufficiently large.
This construction is perfectly similar to many others in \cite{Bi-Co:cob1}
but for completeness we review it rapidly here.
We define  the relevant moduli space by again using an almost complex structure that projects holomorphically away from a compact set.
An application of the
open mapping theorem together with the fact that
 $V_{s}$ is independent of $s$ around the bottlenecks  shows rapidly that the projections of the curves $u$ onto $\C$ remain inside a compact set $\subset \C$. Combined with the fact that the Hamiltonians $H_{s,t}$ have a finite variation on $M$, we obtain a priori energy bounds and this
implies the compactness of the relevant moduli space, up to Floer splitting.
Finally, regularity is achieved by standard perturbative methods (see possibly  \cite{Bi-Co:cob2}).
In summary, we deduce that the map $\eta_{\Psi}$ is a chain map given by a matrix:
$$\eta_{\Psi}: CF(g_{0}(L), g_{0}(L))_{P} \oplus CF(g_{0}(L),g_{0}(L)) _{Q}\to CF(g_0(L), g_0(L))_{P} \oplus CF(g_0(L), g_1(L))_{Q}$$
\begin{equation}\label{eq:matrixphi}
\begin{pmatrix}
 1 & 0 \\
 \mu & \phi_\mathbf{g}
\end{pmatrix}.
\end{equation}
The indexes $P$ and $Q$ in this formula indicate to which one of the two bottlenecks correspond the respective complexes.
The one fact that is remarkable here is that $\phi_{\mathbf{g}}$ appears
in the lower right corner of this matrix - this is due directly to the moving boundary definition of $\phi_{\mathbf{g}}$.
The equation $D_{2}\eta_{\Psi}=\eta_{\Psi} D_{1}$ implies that
$\mu$ is a chain homotopy between $\phi_{V}$ and $\phi_{\mathbf{g}}$ and concludes the proof.
\qed

\subsection{The Corollary \ref{thm:main}: commutativity of the top
square in (\ref{eq:main-diag}).}\label{subsec:top}
We only provide a sketch of the proof of the commutativity of the top square in (\ref{eq:main-diag}) as the result is known by experts.
A variant of this commutativity,
with the Donaldson category of $M$ in the place of $T^{S}D\fuk^{d}(M)$, is contained
in \cite{Hu-La-Le:monodromy}.

\begin{prop} The top square in diagram (\ref{eq:main-diag}) commutes.
\end{prop}
\begin{proof}
Fix $\mathbf{g}=\{g_{t}\}$, $g_{t}\in Ham(M)$, a loop of Hamiltonian
diffeomorphisms. Thus, $g_{0}=g_{1}=\id$. By the description in
\S\ref{sec:ingred} of the various functors involved, the desired
commutativity follows if we show that the two natural
transformations $\xi_{\mathbf{g}}$ and $\eta_{{S[\mathbf{g}]}}$ of
$\id|_{\fuk^{d}(M)}$ have the same image in the homology category $H
fun(\fuk^{d}(M),\fuk^{d}(M))$,
\begin{equation}\label{eq:hlgy-eq}
[\xi_{\mathbf{g}}]=[\eta_{{S[\mathbf{g}]}}]\in H fun(\fuk^{d}(M),\fuk^{d}(M))~.~
\end{equation}
The outline of the proof of (\ref{eq:hlgy-eq}) is as follows.  We
first rewrite  both $[\xi_{\mathbf{g}}]$  and
$[\eta_{S[\mathbf{g}]}]$ in terms of moduli spaces formed of curves
$u$ and, respectively, $v$ defined on  surfaces $S''_{r}$ with $k+1$
boundary punctures and one internal puncture and appropriate
asymptotic conditions and boundary conditions (these conditions are
different for the $u$'s compared to those of the $v$'s). Then, in a
second step, we  relate the curves $u$ and $v$ by a geometric
``naturality transformation'' induced by the map $u(z) \to
g^{-1}_{a(z)}(u(z))$, where $a(-)$ is an appropriate map $a=
a_{S''_{r}}: S''_{r} \to [0,1]$.

To proceed we first define the punctured surfaces $S''_{r}$ in more detail. They are just as
the surfaces $S'_{r}$ used in \S\ref{subsubsec:top-funct} b. except that the point $z_{S'_{r}}$ is replaced by a puncture.
We will use the $S''_{r}$'s  as domains of curves $u:S''_{r}\to M$ that satisfy the usual
 perturbed $J$-holomorphic equation and boundary conditions but with an additional property
relative to the internal puncture. More precisely we will assume
that, in the same way as strip-like ends are fixed for the boundary
punctures, there is a universal choice of cylindrical-like ends
around the internal punctures.  The purpose of the choice of cylindrical-like ends is
the following. Suppose that $K:S^{1}\times M\to \R$ is a generic
periodic Hamiltonian and assume $J$ is an almost complex structure
so that $(K,J)$ is regular in the sense of Hamiltonian Floer
homology. We can then write the correct perturbed $J$-holomorphic
equation for $u:S''_{r}\to M$ by using the strip-like ends around
the boundary punctures , as well as the Floer data as in the
construction of the Fukaya category and using the Hamiltonian Floer
data $(K, J)$ around the internal puncture.  We use this setup to
define two pre-natural transformations of the identity on
$\fuk^{d}(M)$ associated to a Floer cycle $\beta \in CF(K,J)$. The
first one is denoted by $\tilde{\eta}_{\beta}$ and is defined as
follows. We first define
 $$\tilde{\eta}^{k} :CF(L_{1},L_{2})\otimes\ldots \otimes CF(L_{k},L_{k+1}) \otimes CF(K,J)\to
 CF(L_{1},L_{k+1}),$$  given by counting curves $u$ as before
 with boundary conditions so that  $u(C_{i})\subset L_{i}$ where $L_{1}, \ldots, L_{k+1}\in \mathcal{O}b(\fuk^{d}(M))$ and with asymptotic conditions corresponding to the generators of $CF(L_{i},L_{i+1})$ for the $i$-th
 boundary puncture and to a generator of $CF(K,J)$ for the internal puncture. We then
 put $\tilde{\eta}_{\beta}= \tilde{\eta}(-,\ldots, -,\beta)$.

The second pre-natural transformation is denoted
$\tilde{\xi}_{\beta,\mathbf{g}}$ and its definition involves also a
fixed loop inside $Ham(M)$, $\mathbf{g}=\{g_{t}\}$, $t\in [0,1]$,
with $g_{0}=g_{1}=\id$. The definition is exactly as in the case of
$\tilde{\eta}$ with the difference that the boundary conditions are
now ``moving'', as in the definition of $\xi_{\mathbf{g}}$ in
\S\ref{subsubsec:top-funct}. In other
 words, in this case we have that along the segment $C_{j}$ of the boundary of $S''_{r}$, $u(e^{-2 i \pi s})\in g_{s}(L_{j})$ for $s_{j}\leq s\leq s_{j+1}$ (if we use a parametrization with the output at $(1,0)$ and the internal
 puncture at $(0,0)$).

Next we compare $\xi_{\mathbf{g}}$ and $\tilde{\xi}_{PSS([M]),
\mathbf{g}}$ where $[M]$ is the fundamental class of $M$ represented
as the maximum of a Morse function $f:M\to \R$ and $PSS([M])\in
CF(K,J)$ is the image of this maximum by the PSS morphism (we assume
a generic metric fixed on $M$ so that the negative gradient of $f$
is Morse-Smale).  To compare these two transformations we notice
that the moduli spaces of curves $u:S'_{r}\to M$ used to define
$\xi_{\mathbf{g}}$ can be thought as well as being pairs $(u, l)$
where $l$ is a negative gradient flow line of $f$ that joins the
maximum of $f$ to the point $u(z_{S'_{r}})$. Because of this we can
use the PSS-method  to define some interpolating moduli spaces that
define a prenatural transformation $\tau$ with the property that
$\mu^{1}(\tau)= \xi_{\mathbf{g}}- \tilde{\xi}_{PSS([M]),
\mathbf{g}}$ in the functor $A_{\infty}$-category
$fun(\fuk^{d}(M),\fuk^{d}(M))$ (to be more precise we should work
here with the free category $\fuk^{d}(M)^{free}$ but we leave these
details to the reader). In other words, the homology classes
$[\xi_{\mathbf{g}}]$ and $[\tilde{\xi}_{PSS([M]), \mathbf{g}}]$
coincide.

 We proceed in a perfectly similar way to show that, when $\alpha\in C(f)$ is a cycle,
 then $\eta_{\alpha}$ is homologous to $\tilde{\eta}_{PSS(\alpha)}$.

 To end the proof we need to show that \begin{equation}\label{eq:hlgical-id2}[\tilde{\xi}_{PSS([M]), \mathbf{g}}]=
 [\tilde{\eta}_{PSS(S(\mathbf{g}))}].
 \end{equation}
 From the discussion above we know already that, in homology, both natural transformation do not depend on the choice of Hamiltonian $K$.

To show (\ref{eq:hlgical-id2}), denote the moduli spaces used in the
definition of $\tilde{\xi}_{PSS([M]), \mathbf{g}}$ by
$\mathcal{M}_{K}$ and the moduli spaces used in the definition of
$\tilde{\eta}_{PSS(S(\mathbf{g}))}$ by $\mathcal{M}_{K}'$.
Certainly, these moduli spaces depend on various data and choices
that we omit from the notation. The important point is that we will
define a homeomorphism $\phi : \mathcal{M}_{K}\to \mathcal{M'}_{H'}$
where $H'(t,x)=G(t, g_{t}(x)) + K(t,g_{t}(x))$ and $G$ generates
$\mathbf{g}$, so that our transformation transports admissible data
on one side to admissible data on the other.
 We will denote the elements of $ \mathcal{M}_{K}$ by $u$ and the elements of $\mathcal{M'}_{H'}$ by
$u'$. The expression of $\phi(u)$ is
\begin{equation}\label{eq:nat}
\phi(u)(z)=g^{-1}_{a_{S''_{r}}(z)}(u(z)),
\end{equation}
 where $a_{S''_{r}}:S''_{r}\to [0,1]$
are functions so that:
\begin{itemize}
\item[i.] On each boundary component $C_{i}$ of $\partial S''_{r}$
the moving boundary conditions along $g_{s}(L_{i})$ are transformed to fixed ones along $L_{i}$.
\item[ii.] The asymptotic orbit $\gamma$ corresponding to the internal puncture is transformed
by $\gamma(s)\to g^{-1}_{s}(\gamma(s))$.
\end{itemize}
Of course, these maps $a_{S''_{r}}$ have to be defined also in a way
coherent with gluing and splitting of the Deligne-Mumford-Stasheff
associahedron. It is easy to see that such a system of functions
$a_{S''_{r}}$ can be found. For instance, assuming that $S''_{r}$ is
a disk with $k+1$ boundary punctures parametrized so that the output
is at the point $(1,0)$ and the internal puncture is at the point
$(0,0)$, then we can take $a_{S''_{r}}(r e^{-2 i \pi s})=s$.  It is
well-known that a transformation as in (\ref{eq:nat}) takes elements
in the moduli space $\mathcal{M}$ to elements in $\mathcal{M}'$
which are defined in terms of the ``transported'' almost complex
structures, Floer data, etc. In view of the definition of
$S(\mathbf{g})$ from \S \ref{subsubsec:top-funct}, this shows that,
with choices of data related by a transformation as in
(\ref{eq:nat}), the two natural transformations
$\tilde{\xi}_{PSS([M]), \mathbf{g}}$ and
 $\tilde{\eta}_{PSS(S(\mathbf{g}))}$ agree even at the chain level and this implies (\ref{eq:hlgical-id2}).
 \end{proof}

% !TEX root = Serep.tex
\section{Some calculations and examples.}\label{sec:examples}
The main purpose of this section, in \S\ref{subsec:toric-cal} is to provide a class of
Lagrangians $L$ so that the monoid $\mor_{\cob^{d}_{0} }(L,L)$ (the
operation is concatenation of cobordisms) is nontrivial and to get
some ``lower bound'' estimate for its size. 
Notice that the morphisms in $\cob^{d}_{0}$ from $L$ to $L$ are
represented by elementary cobordisms  ($L$ is viewed here as a
family with a single element).

Along the way, we will
make explicit some of the maps in diagram (\ref{eq:main-diag}).
Moreover, in \S\ref{subsubsec:flavors-Seidel} 
we summarize seven various descriptions,
(from the literature and/or this paper) of the Lagrangian Seidel morphism in its simplest form which is the Donaldson category version of $\widetilde{S}$ from diagram (\ref{eq:main-diag}), for a fixed Lagrangian
 $L$.

\subsection{Toric calculations.} \label{subsec:toric-cal}
Our ambient manifold $M^{2n}$ is, all along this sub section, toric, Fano
with minimal Chern number at least $2$.  We denote by $G_{M}$ the subgroup
of the group of  units of $QH(M)$ ($QH(M)$ is ungraded here)
generated by the divisors associated to the maximal faces of the Delzant polytope (the definition
will be made explicit in \S\ref{subsec:toric-facts}).

A toric manifold $(M,\omega)$ admits a canonical anti-symplectic
involution $\tau:M\to M$ that leaves the moment map invariant and
that maps orbits of the torus action to orbits.  The fixed point set
of the involution is the {\em real} Lagrangian $L_\R :=
\text{Fix}(\tau)$. It is not difficult to see \cite{Haug:realtoric}
that our assumption that $M$ is Fano with minimal Chern class at
least two implies that $L_{\R}$ is monotone.

\begin{maincor}\label{cor:toric} With the notation and under the assumptions above we have:
\begin{itemize}
\item[i.] Any monotone, non-narrow Lagrangian $L\subset M$
 (thus, so that $QH(L)\not=0$) with $N_{L}\geq 3$ has the property that
 $\mor_{\cob^{d}_{0}}(L,L)$ is nontrivial.
\item[ii.] The image of the monoid morphism
$$\mathcal{F}_{L_{\R}}:\mor_{\cob^{d}_{0}}(L_{\R},L_{\R})\to QH(L_{\R})^{\ast}$$
- obtained by  restriction from the bottom line of diagram (\ref{eq:main-diag}) -
 contains a group isomorphic to  $G_{M}$.
 \end{itemize}
\end{maincor}

To exemplify point ii, we immediately deduce from the definition of $G_{M}$
 that for the standard real projective space $\R P^{n}\subset \C P^{n}$
the monoid $\mor_{\cob^{d}_{0} }(\R P^{n},\R P^{n})$ contains  an element $u$
so that $\{1, u,\ldots, u^{i}, \ldots, u^{n}\}$ are pairwise distinct.

\

The proof is based on the specialization of diagram
(\ref{eq:main-diag}) to a single Lagrangian $L\in
\mathcal{L}^{\ast}_{d}$ as below:

\begin{eqnarray}\label{eq:diag-ex}
   \begin{aligned}
   \xymatrix{
     \pi_{1}(Ham(M))\ar[rrr]^{S} \ar[d]_{i}
     & &  &QH(M)^{\ast}\ar[d]^{\ast [L]}\\
   \Pi (Ham(M))\ar[d]_{\Sigma}\ar[rrr]^{\widetilde{S}(-)_{L}} & & & HF (g_{0}(L),g_{1}(L)) \ar[d]^{id}\\
    \mor_{\cob^{d}_{0}(M))}(g_{0}(L),g_{1}(L))\ar[rrr]_{\mathcal{F}_{L}}& & & HF(g_{0}(L),g_{1}(L))}
\end{aligned}
\end{eqnarray}

Here  $g_{0}$, respectively $g_{1}$,  are the ends of the path
$\mathbf{g}\in \Pi (Ham(M))$. We will only restrict to examples with
$g_{0}=id$ and $g_{1}(L)=L$. In this case $HF(g_{0}(L), g_{1}(L))$
is canonically identified with $HF(L,L)=QH(L)$ - this explains  the
vertical arrow originating in the upper-most right corner. In a
different context, Lagrangian loops of this type have been also
studied by Akveld-Salamon \cite{Ak-Sa:Loops}. The morphism
$\mathcal{F}_{L}$  associates to each cobordism $V:L\cobto L$ a
quantum homology class $\phi_{V}([L])\in HF(L,L)=QH(L)$  and was
described explicitly at the beginning of section \S\ref{sec:ftilde}.
Given that in our case $V$ is elementary, the image of
$\mathcal{F}_{L}$ is included in the invertibles of $QH(L)$.

\

Besides diagram (\ref{eq:diag-ex}),  as we shall see, the proof of Corollary \ref{cor:toric}  follows rapidly from
recent results of Haug \cite{Haug:realtoric} and Hyvrier \cite{Hy:circle}
as well as older work of McDuff-Tolman \cite{McD-Tol:circle}.
We will review the ingredients needed from these works in the first two subsections below,
\S \ref{subsec:toric-facts} and \S \ref{sec:realtoric}.
The proof of Corollary \ref{cor:toric} appears in  \S\ref{subsubsec:proof-cor}.

\

In the arguments that follow we make use of the identification
$QH(L) \cong HF(L,L)$ via the $PSS$ isomorphism and in the
calculations below we work with coefficients in the Novikov ring
$\Lambda = \mathbb{Z}_2[t, t^{-1}]$, $\deg t = -1$. This is
legitimate because we only consider a single Lagrangian at a time.
Similarly, the coefficient ring for the quantum homology of the
ambient symplectic manifold will be the subring of even powers
$\Lambda_{ev} = \mathbb{Z}_2[t^2, t^{-2}]$. In case we need to
distinguish the ungraded quantum homology from the graded one we
indicate explicitly the coefficients: $QH(M;\Z_{2})$ stands for the
ungraded version and $QH(M;\Lambda_{ev})$ stands for the graded one.
One can obviously pass from the graded version to the ungraded one
by setting $t=1$ and forgetting the degrees.

\subsubsection{Elements of toric topology}\label{subsec:toric-facts}
We review rapidly some toric basic facts
 following closely McDuff-Salamon's presentation from
\cite{McD-Sa:Jhol-2}, Chapter 11.4.  Our manifold
is endowed with an effective
Hamiltonian action of an $n$-torus and there is an associated moment
map $\mu: M \to (\mathbb{R}^n)^*$ whose image, the moment polytope,
is determined by vectors $v_i \in \mathbb{R}^n \cap \Z^n$, $i = 1,
..., d$, and real numbers $\{ a_i \}_{i=1}^d$ via
$$\Delta := \mu(M) = \{ f \in (\R^n)^* \; | \; f(v_i) \geq a_i \}.$$

The codimension one faces of $\Delta$ are assumed non-empty and are
given by
$$F_i = \{ f \in \Delta \; | \; f(v_i) = a_i \}.$$
Denote by $[D_i] \in H_{2n - 2}(M)$ the fundamental class of the
submanifold $D_i := \mu^{-1}(F_i)$.

%The (possibly empty) face of the polytope associated to a subset $I
%\subset \{1, ..., d \}$ is defined by $\Delta_I = \cap_{i \in I}
%F_i$.  Such a subset $I$ is called primitive if $\Delta_I =
%\emptyset$ and $\Delta_{I'} \neq \emptyset$ for any proper subset
%$I' \subset I$.  To any primitive subset $I$ corresponds a unique
%vector $p_I = (p_1, ..., p_d) \in \Z^d$ such that
%$$\sum_{i=1}^d p_i v_i = 0, \quad p_i = 1 \; (i \in I), \quad p_i \leq 0 \; (i \notin I).$$

\

The
quantum homology ring of $M$ is generated by the classes
$[D_i]$:
\begin{equation}\label{eq:toric-qhlgy}
QH(M; \Lambda_{ev}) \cong \frac{\Z_2 [ [D_1], ..., [D_d] ][t^2, t^{-2}]}{(P(\Delta), SR_Q(\Delta))}.
\end{equation}
We refer to \cite{McD-Sa:Jhol-2} for the description of the ideals
$P(\Delta)$ and $SR_Q(\Delta)$.

% are defined by
%$$P(\Delta) = \Big\langle \sum_{i=1}^d f(v_i)[D_i] \; | \; f \in (\Z^n)^* \Big\rangle,$$
%$$SR_Q(\Delta) = \Big\langle \prod_{i \in I} [D_i] = t^{ 2 \Sigma_{i=1}^d p_i} \prod_{i \notin I}[D_i] \; | \; I \text{ is primitive}\Big\rangle.$$

\

The classes $[D_{i}]$ are easily seen to be invertible. We denote by $G_{M}$ the subgroup
they generate inside the group of invertible elements $QH(M;\Z_{2})^{\ast}\subset QH(M;\Z_{2})$.

\

 The relation with Seidel's representation  is provided by a result
of McDuff-Tolman \cite{McD-Tol:circle}: there is a family of loops
$\mathbf{G}_i \in \pi_1(Ham(M))$ generated by Hamiltonians
$\mu_{v_i}: M \to \R, \quad \mu_{v_i}(m) := \mu(m)(v_i)$, whose
Seidel elements are given by
$$S(\mathbf{G}_i) = [D_i] t^{-2}.$$

For any Lagrangian $L\in \mathcal{L}^{\ast}_{d}$
we deduce from diagram (\ref{eq:diag-ex})  that
\begin{equation}\label{eq:action}
\mathcal{F}_{L}( ([0,1]\times L)^{\mathbf{G_{i}}^{-1}})=[D_{i}]\ast [L]~.~
\end{equation}

\subsubsection{Properties of $L_{\R}$}\label{sec:realtoric}
The intersection
$L_\R \cap D_i$ is non-transverse but it is a codimension one submanifold $x_i$ of
$L_\R$. Set $[x_i] := [L_\R \cap D_i] \in H_{n-1}(L_\R)$.

\

Haug \cite{Haug:realtoric} showed that there is an
isomorphism of rings that doubles the degree and is defined by
$$\beta_{\R}: QH_*(L_\R; \Lambda) \to QH_{2*}(M; \Lambda_{ev})$$
$$[x_i] \mapsto [D_i]$$
$$t \mapsto t^2.$$

The $QH(M)$-module structure on $QH(L_\R)$ has been
computed by Hyvrier \cite{Hy:circle} and is completely determined by
$$[D_i] * [L_\R] =[x_i]^2,$$
since the ring structures on $QH(M)$ and $QH(L_\R)$ are generated
respectively by $[D_i] \otimes \Lambda_{ev}$ and $[x_i] \otimes
\Lambda$. Hyvrier \cite{Hy:circle} also shows that there is a Lagrangian analogue of the loops
$\mathbf{G}_i$, given by the paths $\mathbf{g}_i(t) :=
\mathbf{G}_i(t/2)$. These paths still satisfy $\mathbf{g}_i(1)(L_\R) =
L_\R$ and he shows that the associated relative Seidel
elements are given by
$$\widetilde{S}(\mathbf{g}_i)(L_\R) = [x_i] t^{-1}.$$

\subsubsection{Proof of Corollary \ref{cor:toric}}\label{subsubsec:proof-cor}
For the first point we notice that equation (\ref{eq:action}) directly implies the claim
as soon as we show that $$[D_{i}]\ast [L]\not =[L]t^{2}\in QH(L;\Lambda)~.~$$
But in view of the assumption that $N_{L}\geq 3$ and given that $|t|=-1$, by possibly
using the pearl model for Lagrangian quantum homology, this inequality is immediate.

We proceed to show the second point.
Let $w=Y_{1}^{i_{1}}\ldots Y_{d}^{i_{d}}$ be a word written in the letters $Y_{i}$, $1\leq i \leq d$.
We view $Y_{i}$ as $Y_{i}=x_{i}t^{-1}$.
We associate to this word an element in $\mor_{\cob^{d}_{0}}(L_{\R},L_{\R})$ as follows.
We put $V_{0}=[0,1]\times L_{\R}$ and write:
$$\Xi(w)=(V_{0}^{\mathbf{g}_{d}^{-1}}) ^{\ast^{{i_{d}}}}\ast \ldots\ast (V_{0}^{\mathbf{g}_{1}^{-1}})^{\ast^{ i_{1}}}~.~$$
Here $\ast$ stands for concatenation. In short,  $\Xi(w)$ is
obtained by concatenating from right to left the suspensions
$V_{0}^{\mathbf{g}^{-1}_{i}}$ in the number and order prescribed by
the word $w$. The calculations in \S \ref{sec:realtoric}  together
with the commutativity of diagram (\ref{eq:diag-ex}) imply that
$\mathcal{F}_{L_{\R}}(\Xi(w))=[w]\in  QH(L_{\R};\Z_{2})$ where $[w]$
is now written in the $[x_{i}]$'s. Given that $QH(L_{\R})$ is
isomorphic to $QH(M)$ as a ring, we deduce that the image of
$\mathcal{F}_{L_{\R}}$ contains the subgroup
$\beta_{\R}^{-1}(G_{M})$. \qed

\subsection{The many flavors of the relative Seidel morphism}\label{subsubsec:flavors-Seidel}

We shall discuss seven different, yet homologically equivalent, definitions of the relative Seidel morphism.  
Let $\mathbf{g}$ be a path of Hamiltonian diffeomorphisms with ends $g_0$ to $g_1$. We concentrate on a single Lagrangian $L\in \mathcal{L}^{\ast}_{d}$. Each one of the seven definitions gives a different
description of the same  homology class $\in HF(g_{0}(L), g_{1}(L))$ which is called the Lagrangian (or relative) Seidel element. 
 
 \
 
\noindent A.   The first definition - originating in \cite{Se:book-fukaya-categ} - and recalled in \S \ref{subsubsec:top-funct} is defined as the element $\xi^0_{\mathbf{g}} \in CF(g_0(L), g_1(L))$  given by counting curves $u$ with a single boundary puncture mapped to an intersection point of $g_0(L)$ and $g_1(L)$, satisfying the moving boundary condition $u(e^{-2 \pi i t}) \in g_t(L)$.

\

\noindent B.  In Remark \ref{rem:movingboundarydef} we gave an alternative definition of the element at A as $\phi_{\mathbf{g}}(PSS([g_0(L)]))$,  where $\phi_{\mathbf{g}}: CF(g_0(L), g_0(L)) \to CF(g_0(L), g_1(L))$ is the moving boundary morphism.

\

\noindent C.  The chain morphism $\phi_{\mathbf{g}}$ from point B is chain homotopic
 -  as is shown in Lemma \ref{lem:moving} - to the ``cobordism'' morphism $\phi_{\Sigma(\mathbf{g})(L)}: CF(g_0(L), g_0(L)) \to CF(g_0(L), g_1(L))$ as constructed in \S \ref{sec:ftilde}.

\

\noindent D. When $g_{0}=id$, there is a definition of a morphism, chain equivalent to the one at B, based on the same type of naturality transformation as described in \S\ref{subsubsec:top-funct} but starting from Lagrangian Floer homology - with the differential counting ``semitubes''. This morphism was used
 in \cite{Leclercq:spectral}.

\

\noindent E.  Given a Morse-Smale function $f: L \to \R$, we denote by $\mathcal{C}(f)$ the associated pearl complex (we omit the almost complex structure and the Riemannian metric from the notation) \cite{Bi-Co:rigidity}.  Assuming that $g_0= \id$ and $g_1(L) = L$, one then defines $\widetilde{S}_\mathcal{P}(\mathbf{g}): \mathcal{C}(f) \to \mathcal{C}(f)$ by counting pearly trajectories with the property that exactly one J-holomorphic disc verifies a moving boundary condition $u(e^{-2 \pi i t}) \in g_{a(t)}(L)$.  The function $a:[0,1] \to [0,1]$ is monotone, satisfies $a|_{[0, 5/8]} = 0$ and $a|_{[7/8, 1]} = 1$ (this reparametrization is chosen so that the disc moves along $g_t(L)$ on its upper half part).  It was shown in \cite{Cha:thesis} that this morphism is chain-homotopic to the morphism at point B.

\

\noindent F.  Still assuming that $g_0 = \id$ and $g_1(L) = L$, Hu and Lalonde \cite{Hu-La:Seidel-morph} considered a Hamiltonian fibration $M \to P \to D^2$, defined using $\mathbf{g}$, in which the mapping torus $N:= \{ g_t(L) \}$ appears as a Lagrangian submanifold.  The restriction of this fibration to $\partial D^2$ gives by construction $L \to N \to S^1$.  By using holomorphic sections $(D^2, S^1) \to (P, N)$, they defined a morphism $\mathcal{C}(f) \to \mathcal{C}(f)$ that induces an isomorphism in Lagrangian quantum homology.  It was also shown in \cite{Cha:thesis} that this morphism coincides in homology with $[\widetilde{S}_{\mathcal{P}}(\mathbf{g})]$.  The proof of this fact uses a result of \cite{Ak-Sa:Loops}, in which they prove that such holomorphic sections  correspond to disks in $M$ with moving boundary condition $u(e^{-2 \pi i t}) \in g_t(L)$, satisfying a Hamiltonian perturbation of the $\overline{\partial}_J$-equation.  One can then count pearly trajectories using exactly one such disc, and a homotopy sending the perturbation to zero shows that this gives the same count as $\widetilde{S}_{\mathcal{P}}(\mathbf{g})$ in homology.

\

\noindent G. Finally, when $\mathbf{g}$ is a loop based at the identity, the Lagrangian Seidel element
described before descends from the Seidel element  of $\mathbf{g}$ in $M$ via the quantum module action
so that, using for instance the description at B (which is possibly the simplest), we have:
$$\phi_{\mathbf{g}}([L])= S(\mathbf{g})\ast [L]~.~$$
The identity D=F was first shown in \cite{Hu-La-Le:monodromy}.  It is obviously reflected in the commutativity of the top of the diagram (\ref{eq:main-diag}).

\bibliographystyle{alpha}
\bibliography{bibliography}

\end{document}